\documentclass[11pt,a4paper]{amsart}

\usepackage[latin1]{inputenc}
\usepackage{a4wide}
\usepackage{mathrsfs}
\usepackage{amsmath}
\usepackage{amsfonts}
\usepackage{amssymb}
\usepackage{amsthm}
\usepackage{graphicx}
\usepackage{epstopdf}
\usepackage{todonotes}
\usepackage{comment}
\usepackage{soul}
\usepackage{bbm}
\title[$L^2$-Exponential decay of Langevin dynamics]{Weighted $L^2$-contractivity of Langevin dynamics with singular potentials}\author{Evan Camrud}
\address{Department of Mathematics\\Iowa State University\\Ames, IA USA\\ email:\texttt{ecamrud@iastate.edu}}
\author{David P. Herzog}
\address{Department of Mathematics\\Iowa State University\\Ames, IA USA\\ email:\texttt{dherzog@iastate.edu}}
\author{Gabriel Stoltz}
\address{CERMICS\\Ecole des Ponts\\ Marne-la-Vall\'{e}e\\ France \& MATHERIALS team-project\\ Inria Paris, France\\ email:  \texttt{gabriel.stoltz@enpc.fr}}
\author{Maria Gordina}
\address{Department of Mathematics \\University of Connecticut\\ Storrs, CT USA \\email:  \texttt{maria.gordina@uconn.edu}}

\newtheorem{Lemma}{Lemma}
\newtheorem{Theorem}{Theorem}\newtheorem{Proposition}{Proposition}

\newtheorem{Corollary}{Corollary}

\newtheorem{Definition}{Definition}
\newtheorem{Remark}{Remark}
\newtheorem{Assumption}{Assumption}

\newcommand{\eqn}[1]{\begin{displaymath}\begin{split}#1\end{split}\end{displaymath}}
\newcommand{\eqnn}[1]{\begin{equation}\begin{split}#1\end{split}\end{equation}}
\newcommand{\R}{\mathbf{R}}

\newcommand{\N}{\mathbf{N}}
\newcommand{\E}{\mathbf{E}}

\newcommand{\X}{\mathscr{X}}
\newcommand{\PP}{\mathbf{P}}

\newcommand{\B}{\mathcal{B}}

\renewcommand{\leq}{\leqslant}
\renewcommand{\geq}{\geqslant}
\newcommand{\rme}{\mathrm{e}}
\newcommand{\dps}{\displaystyle}
\newcommand{\sP}{\mathscr{P}}
\newcommand{\sPp}{\mathscr{P}^\perp}

\begin{document}
\maketitle

\begin{abstract}
  Convergence to equilibrium of underdamped Langevin dynamics is studied under general assumptions on the potential~$U$ allowing for singularities. By modifying the direct approach to convergence in~$L^2$ pioneered by F.~H\'erau and developed by Dolbeault, Mouhot and Schmeiser, we show that the dynamics converges exponentially fast to equilibrium in the topologies~$L^2(d\mu)$ and~$L^2(W^* d\mu)$, where $\mu$ denotes the invariant probability measure and $W^*$ is a suitable Lyapunov weight. In both norms, we make precise how the exponential convergence rate depends on the friction parameter~$\gamma$ in Langevin dynamics, by providing a lower bound scaling as~$\min(\gamma, \gamma^{-1})$. The results hold for usual polynomial-type potentials as well as potentials with singularities such as those arising from pairwise Lennard-Jones interactions between particles.  
\end{abstract}

\section{Introduction}

Langevin dynamics, and related stochastic differential equations dictating the evolution of physical systems at the atomistic scale, have long been a subject of interest in molecular dynamics simulation~\cite{hoogerbrugge1992simulating,cances2007theoretical,tuckerman2010statistical,durrant2011molecular,LM15} and more recently in machine learning~\cite{welling2011bayesian,teh2016consistency,li2016preconditioned,ChandraRohitash2017BNLv}. In these settings, estimates on the rate of convergence of the dynamics to equilibrium provide a rough measure for how long to run the corresponding algorithm in order to take random samples from the target Boltzmann--Gibbs distribution -- as quantified by bounds on the asymptotic variance of averages over trajectories, dictated by a Central Limit Theorem~\cite{Bhattacharya}. This has led to a number of works analyzing the precise nature of this convergence under various assumptions on the potential~$U$, both in the literature on statistical physics and in kinetic theory (where the typical dynamics is similar to the Fokker--Planck equation ruling the evolution of the law of the solutions of Langevin dynamics).

The first known result on the convergence of Langevin dynamics is due to Tropper~\cite{tropper1977ergodic}, who proved that the dynamics mixes when the Hessian of the potential is bounded. Later, the works of Wu~\cite{Wu_01}, Mattingly, Stuart, and Higham~\cite{MSH_02} and Talay~\cite{Talay_02} improved these results by proving convergence to equilibrium at an exponential rate for `polynomial-like' potentials, in total variation. Their estimates relied on the existence of a Lyapunov function of the form~$H(q,p)+c q\cdot p$ for a small constant $c>0$, where 
\begin{align}
\label{E:Ham}H(q,p)=\frac{|p|^2}{2}+U(q)
\end{align}
denotes the Hamiltonian of the system.  Additionally, in Talay's work~\cite{Talay_02}, exponential convergence in Sobolev spaces~$H^k(d\mu)$ for $k\geq 1$ as well as in weighted Sobolev spaces was obtained. 

These papers, as well as other seminal works such as~\cite{EH03,HN04,HN05,MN06}, later motivated the development of hypocoercivity by Villani~\cite{Villani_2009} where, in the case of Langevin dynamics, exponential convergence to equilibrium was established in $H^1(d\mu)$ for $C^2(\R^d)$ potentials for which~$\mu$ satisfies a Poincar\'e inequality, and such that the global bound 
\begin{align}
\label{eqn:1v2}
\left|\nabla^2 U\right|\leq C(1+|\nabla U|)
\end{align}
holds for some constant $C>0$.  There, `twisted gradients' are used to obtain contractive properties in $H^1(d\mu)$ which are ultimately transfered to $L^2(d\mu)$ by hypoelliptic regularization~\cite{Herau07,HP08}. Using a heuristic overdamped scaling analysis, and extending the convergence result by~H\'erau~\cite{Herau06}, Dolbeault, Mouhot and Schmeiser~\cite{DMS09,dolbeault2015hypocoercivity} established a convergence result similar to the one by Villani under the condition~\eqref{eqn:1v2}, but using a direct perturbative $L^2(d\mu)$ approach -- that is, instead of first obtaining a contraction in $H^1(d\mu)$ and then transferring it to~$L^2(d\mu)$, one can directly obtain the contraction in~$L^2(d\mu)$ by adding an appropriate perturbation to the norm.       

Although the growth condition~\eqref{eqn:1v2} is general, it does not allow for potentials with singularities. This is intuitively the case because the growth of the singular terms in $U$ increases after taking derivatives, so that~\eqref{eqn:1v2} cannot be satisfied. In this direction, the papers of Conrad and Grothaus~\cite{conrad2010construction} and Grothaus and Stilgenbauer~\cite{grothaus2015hypocoercivity} established a polynomial rate of convergence for time averages (along realizations of the stochastic dynamics) to ensemble averages (with respect to $\mu$) for potentials that include Lennard-Jones interactions. Later, adopting a Lyapunov approach and denoting by~$\mathscr{O}$ the open domain on which the potential is finite, exponential convergence to equilibrium was obtained in~\cite{CHMMS_17,HerMat_19} in a weighted total variation distance for a class of potentials belonging to~$C^\infty(\mathscr{O})$ satisfying the general condition: For all $\varepsilon>0$, there exists $C_\varepsilon>0$ such that
\begin{align}
\label{eqn:2v1}
\left|\nabla^2 U\right|\leq \varepsilon |\nabla U|^2+C_\varepsilon,
\end{align}
holds globally. In particular, it can be shown that this condition is satisfied by systems of~$N$ particles with positions~$q_1,\dots,q_N$, interacting via pairwise singular potentials such as the Lennard--Jones' potential
\begin{align*}
    U(q)=\sum_{1 \leq i < j \leq N}\left(\frac{A_{ij}}{|q_i-q_j|^{12}}-\frac{B_{ij}}{|q_i-q_j|^{6}}\right),
\end{align*}
or the Coulomb interaction in dimension $d\geq3$
\begin{align*}
    U(q)=\sum_{1 \leq i < j \leq N}\frac{A_{ij}}{|q_i-q_j|};
\end{align*}
see~\cite[Appendix]{BGH_19}. In particular, an explicit Lyapunov function of the form $\exp(\delta H+\psi)$ was constructed in~\cite{CHMMS_17,HerMat_19}, with $\delta$ small and $\psi$ an appropriate lower-order perturbation.  See also~\cite{LuMat_19} for an extension of this result to Coulomb potentials in spatial dimension $d=2$.  Nevertheless, explicit estimates on the convergence rate were not given, especially as they depend on key parameters of the dynamics, \emph{e.g.} the friction $\gamma>0$ or the dimension $d$. This motivated the work~\cite{BGH_19} which aimed at getting an explicit dependence of the rate on the dimension $d$, by making use of a certain weighted~$H^1$ topology where the weight satisfies a Lyapunov-type condition. See also~\cite{CGMZ_19} for a related work.  

The interest of working with the weighted topology $L^2(W^*\,d\mu)$ is that functions in this space have a faster decay at infinity than functions in~$L^2(d\mu)$ (see Remark~\ref{rmk:finer_topo} for a more precise discussion of this point). Exponential convergence in this topology then implies that the solution~$\phi$ to the Poisson equation~$-L\phi = g \in L^2(W^*\,d\mu)$ also belongs to~$L^2(W^*\,d\mu)$, and not just to~$L^2(d\mu)$. Since the solution~$\phi$ decays faster than typical functions in~$L^2(d\mu)$, it is expected that spectral Galerkin discretizations such as the ones considered in~\cite{roussel2018spectral} lead to smaller approximation errors, which turns out to be useful when solutions to Poisson equations are used in numerical methods, for instance to construct control variates~\cite{roussel2019perturbative}.

In terms of the scalings in the prefactors in the convergence bounds, we obtain estimates uniform with respect to the friction, except in the weighted setting when $\gamma \gg1$.  This `break down' makes sense intuitively because, when $\gamma \gg1$, the Lyapunov function we construct is needed in a shrinking part of space as the noise dominates in a growing part of the domain. Better estimates could maybe be obtained by recent approaches which share some similarities with the direct~$L^2$ approach we use here, but do not rely on changing the underlying scalar product, either by using some space-time Poincar\'e inequality as in~\cite{armstrong2019variational,CLW_19}, or by directly seeking estimates on the resolvent of the generator of the dynamics instead of looking for decay estimates of the evolution semigroup~\cite{cances2007theoretical}.

This paper is organized as follows.  In Section~\ref{sec:setting}, we introduce the technical assumptions made throughout the paper and state the main results.  Sections~\ref{sec:genresults} and~\ref{sec:scalings} contain the proofs of the main results assuming a key result, namely Proposition~\ref{prop:key}.  This proposition affords a number of formal manipulations there and is later proved in Section~\ref{sec:auxest} along with some important auxiliary estimates.

\section{Mathematical setting and statement of the main result}
\label{sec:setting}

\subsection{Setting and basic assumptions}

Throughout the paper, we study the Langevin stochastic differential equation
\begin{equation}
  \label{E:StochDiff}
  \left\{ \begin{aligned}
    dq_t&=p_t\, dt,\\
    dp_t&=-\nabla U(q_t)\, dt-\gamma p_t\, dt+\sqrt{2\gamma}\, dB_t .
  \end{aligned} \right.
\end{equation}
In the equation above, the unknowns are the vectors of positions $q_t\in \R^d$ and momentum $p_t\in\R^d$.  The parameter $\gamma>0$ is the friction constant while $B_t$ is a standard $d$-dimensional Brownian motion defined on a probability space $(\Omega, \mathscr{F},\mathbf{P}, \mathbf{E})$. The mapping $U:\R^d\to[0,+\infty]$ is the potential energy function. Note that we allow $U$ to take the value $+ \infty$ in the case when $U$ is singular, as in the situation of a Lennard--Jones or Coulomb interaction force. Throughout, $\mathscr{O}$ will denote the subset of $\R^d$ of positions~$q$ with finite potential energy:
\begin{align*}
  \mathscr{O}=\left\{q\in\R^d:U(q)<\infty\right\},
\end{align*}
while $\X$ denotes the corresponding state space
\begin{align*}
\X=\mathscr{O}\times\R^d
\end{align*}
for the process $x_t:=(q_t, p_t)$.  Depending on the context, we will use either $x$ or $(q,p)$ to denote an element of~$\X$.  

In order guarantee the existence and uniqueness of pathwise solutions for all finite times $t\geq 0$ and all initial conditions $x\in \X$, we make the following assumption on the potential~$U$. Once one makes this assumption, the existence and uniqueness of pathwise solutions follows immediately by first applying the standard iteration procedure to get locally-defined solutions in time; and then extending these local solutions to global ones by using the Hamiltonian $H$ as a basic type of Lyapunov function (see~\cite[Section~2.2]{BGH_19} or~\cite{Has_11, rey-bellet2006ergodic} for further details).

\begin{Assumption}
  \label{assump:basic}
  The potential $U:\R^d\to[0,+\infty]$ satisfies the following properties:
  \begin{itemize} 
  \item[(i)] $U\in C^\infty(\mathscr{O};[0,\infty))$,
  \item[(ii)] The set $\mathscr{O}$ is connected. Moreover, for any $k\in \N$, the open set
    \begin{align*}
      \mathscr{O}_k = \left\{q\in\R^d:U(q)<k \right\}
    \end{align*} 
    has a compact closure.
  \item[(iii)] The integral $\dps \int_{\mathscr{O}} \rme^{-U(q)} \, dq$ is finite.
  \item[(iv)] For any sequence $\{q_k\} \subset \mathscr{O}$ with $U(q_k)\rightarrow \infty$ as $k\rightarrow \infty$, it holds $|\nabla U(q_k)|\rightarrow \infty$ as $k\rightarrow \infty$.
\end{itemize}
\end{Assumption}

Under Assumption~\ref{assump:basic}, the solution of~\eqref{E:StochDiff} is a Markov process and we denote by $\{P_t\}_{t\geq 0}$ the associated Markov semigroup with transition kernel~$P_t(x,dy)$. Letting $\B$ be the Borel $\sigma$-field of subsets of~$\X$, the Markov semigroup acts on bounded, $\B$-measurable functions $\phi: \X\rightarrow \R$ as
\begin{align*}
  \left(P_t \phi\right)(x) = \E_x\left[\phi(x_t)\right] = \int_\X  \phi(y) P_t(x, dy),
\end{align*}
where the expectation~$\E_x$ is with respect to all realizations of the Brownian motion in~\eqref{E:StochDiff} with initial condition~$x_0 = x$. The Markov semigroup also acts on positive, finite $\B$-measures~$\nu$ as 
\begin{align*}
  \forall B \in \B, \qquad \left(\nu P_t\right)(B) = \int_\X \nu(dy) P_t(y, B). 
\end{align*}
Throughout this work, the infinitesimal generator of $\{P_t\}_{t\geq 0}$ is denoted by
\begin{align}
  \label{E:InfGen}
  L=p\cdot\nabla_q-\nabla U(q)\cdot\nabla_p-\gamma p\cdot\nabla_p+\gamma \Delta_p.
\end{align}

Recalling the Hamiltonian $H(q,p)$ of the system~\eqref{E:Ham}, the following $\B$-measure is a probability measure:
\begin{align}
  \label{eqn:invm}
  \mu(dq\,dp)=\frac{1}{\mathscr{N}}\rme^{-H(q,p)}\,dq\,dp,
  \qquad
  \mathscr{N}=\int_{\X}\rme^{-H(q,p)}\,dq\,dp < +\infty,
\end{align}
the finiteness of~$\mathscr{N}$ being guaranteed by Assumption~1(iii).  Moreover, it can be shown that $\mu$ is the unique \emph{invariant} probability measure under the dynamics \eqref{E:StochDiff}; that is, $\mu$ is the unique probability measure satisfying~$\mu P_t = \mu$. Indeed, a short calculation shows that 
\begin{align*}
L^\dagger \left(\rme^{-H}\right)=0,
\end{align*}
where $L^\dagger$ denotes the formal adjoint of $L$ with respect to the $L^2(dx)$ inner product.  Employing Assumption~\ref{assump:basic}, one can then establish~\cite[Proposition~2.5]{HerMat_19} by precisely the same argument given there.  This, in particular, gives the needed hypoelliptic and support properties to ensure uniqueness of $\mu$; see, for example,~\cite{kliemann1987recurrence, rey-bellet2006ergodic}.      

In addition to employing Assumption~\ref{assump:basic}, we will also make use of the following growth assumption on the Hessian of~$U$.  This assumption together with Assumption~\ref{assump:basic} allows for both polynomial and singular potentials~\cite{HerMat_19, BGH_19} and, by the arguments in~\cite{Villani_2009}, implies that $\mu$ satisfies a Poincar\'{e} inequality, as discussed around Proposition~\ref{prop:poincare} below.

\begin{Assumption}
  \label{assump:gc}
  For any $\varepsilon>0$, there exists a constant $C_\varepsilon \in \mathbb{R}_+$ such that
  \begin{equation}
    \label{eq:Ass2}
    \forall q \in \mathscr{O}, \quad \forall y\in \R^d, \qquad \left|\nabla^2 U(q) y\right|\leq \varepsilon|\nabla U(q)|^2|y|+C_\varepsilon|y|.
  \end{equation}
\end{Assumption}

\begin{Remark}
There is a subtle but important difference between the bound in Assumption~\ref{assump:gc} and an often assumed growth condition found in the literature on Langevin dynamics (see, for example, ~\cite{Villani_2009, dolbeault2015hypocoercivity, CLW_19}).  A typical replacement estimate for Assumption~\ref{assump:gc} often reads: There exist constants $C,D \in \mathbb{R}_+$ such that 
\begin{align}
  \label{eqn:strongerbound}
  \forall q \in \mathscr{O}, \qquad \left|\nabla^2 U(q)\right| \leq C |\nabla U(q)| + D.
\end{align}
Such conditions are usually stated in order to control the growth of the potential at infinity, having in mind that the increase in~$U$ should be polynomial in~$|q|$. This is why there is a power~$1$ for the gradient on the righthand side of~\eqref{eqn:strongerbound} while in Assumption~\ref{assump:gc} this exponent becomes $2$ at the expense of a small constant~$\varepsilon$. This difference in the exponent precisely affords the inclusion of potentials with singularities. As a simple illustrative example, consider 
\begin{align*}
\mathscr{O}= (0, \infty)\qquad \text{and} \qquad U(q)=q^{-1}+ q^2.  
\end{align*}
Note that this potential combines a strong repulsion at~$q=0$ with a quadratic confinement at infinity. Now, for $q\geq 1$, an estimate such as~\eqref{eqn:strongerbound} holds since
\[
\forall q \geq 1, \qquad \left|U''(q)\right| \leq 3 \leq 3\left|U'(q)\right| + 3.  
\]
However, near $q=0$, this estimate no longer holds.  Nevertheless, for any $\varepsilon \in (0,1)$, there exists $C_\varepsilon \in \mathbb{R}_+$ such that
\[
\forall q\in (0,1], \qquad \left|U''(q)\right|= 2\left(1 + \frac{1}{q^3}\right) \leq \varepsilon U'(q)^2 + C_\varepsilon. 
\]
Since $|U'(q)|\rightarrow \infty$ as $U(q)\rightarrow \infty$, we finally obtain~\eqref{eq:Ass2} by distinguishing the cases $q \geq 1$ and $q \leq 1$.  
\end{Remark}

\begin{Remark}
Assumption~\ref{assump:gc} is `almost minimal' in the sense that, if $\varepsilon >0$ cannot be chosen as small as wanted, there are smooth functions~$U$ on $\mathscr{O}$ satisfying the bound~\eqref{eq:Ass2} but such that the unnormalized measure $\widetilde{\mu}$ given by 
\begin{align*}
  \widetilde{\mu}(dq \, dp)= \rme^{-H(q,p)}\, dq \, dp
\end{align*}
is no longer finite, hence cannot be normalized into a probability measure. As a concrete example, consider $\mathscr{O}=(0, \infty)$ and any $C^\infty(\mathscr{O})$ function $U$ such that $U(q)= \alpha \log(q)$ for $q\geq 2$, with $\alpha > 0$. Then, 
\begin{align*}
  \forall q\geq 2, \qquad \left|U''(q)\right| = \frac{\alpha}{q^2} = \frac1\alpha |U'(q)|^2,
\end{align*}
Note that $\widetilde{\mu}$ is no longer finite for $\alpha \leq 1$ since $\rme^{-U(q)}=q^{-\alpha}$ for $q\geq 2$.       
\end{Remark}

\subsection{Statement of the main results}
In order to state the main results in this paper, we now introduce some further notation and terminology used throughout this work. 

\subsubsection*{Generators.}
We first introduce the $L^2(d\mu)$-adjoint $L^*$ of $L$. More precisely, for the $L^2(d\mu)$-inner product given by 
\begin{align*}
\langle\phi,\psi\rangle&=\int_\X\phi\,\psi\,d\mu,
\end{align*}
the action of the operator $L^*$ is defined as
\begin{align*}
\forall \psi, \phi \in C_\mathrm{c}^\infty(\X), \qquad \langle L^*\phi,\psi\rangle&=\langle \phi,L \psi\rangle, 
\end{align*} 
where $C_\mathrm{c}^\infty(\X)$ is the set of real-valued, $C^\infty(\X)$ functions with compact support in $\X$. It will sometimes be convenient to decompose $L$ into its symmetric and anti-symmetric parts with respect to the inner product on $L^2(d\mu)$ as
\begin{equation}
  \label{eq:def_L}
  L = L_{\mathrm{H}}+\gamma  L_{\mathrm{OU}},
\end{equation}
where 
\[
L_{\mathrm{H}} = p\cdot\nabla_q-\nabla U(q)\cdot\nabla_p,
\qquad
L_{\mathrm{OU}} = - p\cdot\nabla_p+\Delta_p .
\]
Observe that the Hamiltonian part~$L_{\mathrm{H}}$ of~$L$ is anti-symmetric, while the Ornstein--Uhlenbeck part~$L_{\mathrm{OU}}$ of~$L$ is symmetric, so that 
\[
L^* = -L_{\mathrm{H}}+\gamma  L_{\mathrm{OU}}.
\]
In fact, introducing the momentum reversal operator~$\mathcal{R}$ which is a unitary operator on~$L^2(d\mu)$ acting as $(\mathcal{R}\phi)(q,p) = \phi(q,-p)$, a simple computation shows that
\begin{equation}
  \label{eq:Lstar_RLR}
  L^* = \mathcal{R} L \mathcal{R}.
\end{equation}
It is also useful to introduce the generator of the \emph{overdamped Langevin} dynamics, namely 
\[
L_{\mathrm{OD}} = \Delta_q-\nabla U(q) \cdot\nabla_q = -\nabla_q^* \cdot \nabla_q.
\]
Note that this operator is the infinitesimal generator of the stochastic gradient system
\begin{align}
  \label{eqn:gradient}
  dq_t &= -\nabla U(q_t) \, dt + \sqrt{2} \, dW_t,
\end{align}
where $W_t$ is a standard $d$-dimensional Brownian motion.  

\subsubsection*{Lyapunov functions.}
In this paper, we construct a distance which is contractive (at large times) for the dynamics and equivalent to the norm on $L^2(W^*\,d\mu)$ where $W^*\geq 1$ is a conveniently chosen weight function. In fact, $W^*$ will be some Lyapunov function for~$L^*$, which is why we write~$W^*$ instead of~$W$. We therefore introduce the following terminology.  

\begin{Definition}
  \label{def:Lyap}
  $\quad$
  \begin{itemize}
  \item[(i)]  We call a function $V\in C^2(\X; (0, \infty))$ \emph{strongly integrable} if there exist constants $C >0$ and $\delta \in (0,1)$ such that the following estimate holds on $\X$:
    \begin{align*}
      V + |\nabla_p V| \leq C \rme^{(1-\delta ) H}.  
    \end{align*}  
  \item[(ii)]  We call a strongly integrable function $W\in C^2(\X;(0, \infty))$ a \emph{weak Lyapunov function with respect to $L$ with constants $\alpha, \beta >0$} if    \begin{align*}
    L W&\leq -\alpha W+\beta.
  \end{align*}
  \item[(iii)] We call a strongly integrable function $W^*\in C^2(\X;(0, \infty))$ a \emph{weak Lyapunov function with respect to $L^*$ with constants $\alpha, \beta >0$} if 
    \begin{align*}
      L^* W^*&\leq -\alpha W^*+\beta.
    \end{align*}
  \item[(iv)] We call a weak Lyapunov function $V\in C^2(\X; (0, \infty))$ with respect to $M\in \{ L, L^*\}$ a \emph{strong Lyapunov function with respect to} $M$ if the sub-level sets $\{V\leq c \}$ for $c \in (0,+\infty)$ are compact in~$\X$.  
  \end{itemize}  
\end{Definition}

\begin{Remark}
  \label{rem:wLF}
  Note that any positive constant function is a weak Lyapunov function with respect to both~$L$ and~$L^*$. However, constant functions are not strong Lyapunov functions since their sub-level sets are not compact.  Following the construction in~\cite{HerMat_19}, we show in Theorem~\ref{thm:main} that, under Assumptions~\ref{assump:basic} and~\ref{assump:gc}, there exists a strong Lyapunov function for~$L$.  Hence, by the structure of the generator~$L$, there exists also a strong Lyapunov function for~$L^*$.    
\end{Remark}
  
Throughout, for any operator $M\in \{ L, L^*\}$, we use the notation $\mathscr{W}_{\alpha, \beta}(M)$ to denote the set of weak Lyapunov functions with respect to $M$ with constants $\alpha, \beta>0$, and 
\begin{align*}
  \mathscr{W}(M) = \bigcup_{\alpha, \beta >0} \mathscr{W}_{\alpha, \beta}(M).
\end{align*}
We similarly let $\mathscr{S}_{\alpha, \beta}(M)$ denote the set of strong Lyapunov functions with respect to $M$ with constants $\alpha, \beta >0$ and
\begin{align*}
    \mathscr{S}(M) = \bigcup_{\alpha, \beta >0} \mathscr{S}_{\alpha, \beta}(M).
\end{align*}
The equality~\eqref{eq:Lstar_RLR} implies a clear relationship between $\mathscr{W}(L)$ and $\mathscr{W}(L^*)$, and hence between $\mathscr{S}(L)$ and $\mathscr{S}(L^*)$, as summarized in the following proposition.

\begin{Proposition}
  \label{prop:relLyap}
  The function $W$ is a weak (resp. strong) Lyapunov function for $L$ with constants $\alpha, \beta>0$ if and only if $W^*$ defined by $W^*(q,p)=W(q,-p)$ is a weak (resp. strong) Lyapunov function with respect to $L^*$ with constants $\alpha, \beta >0$.
\end{Proposition}

\begin{Remark}
  The existence of a Lyapunov function is often assumed in the literature to establish geometric ergodicity for SDEs in a (weighted) total variation distance; see, for example, \cite{meyn1993stability,hairer2011yet,rey-bellet2006ergodic}.  Using the one-to-one correspondence between $\mathscr{W}(L)$ and $\mathscr{W}(L^*)$ in our context, we will see here that contraction for the dynamics~\eqref{E:StochDiff} in a weighted $L^2( d\mu)$ sense also follows from the existence of a weak Lyapunov function.  
  \end{Remark}
\begin{Remark}
  \label{rmk:finer_topo}
 One can always take our weak Lyapunov function to be a positive constant.  However, for potentials satisfying Assumptions~\ref{assump:basic} and~\ref{assump:gc} (see~\cite{BGH_19,HerMat_19,LuMat_19}), the set $\mathscr{S}(L)$ is non-empty, hence so is $\mathscr{S}(L^*)$ by this one-to-one correspondence.  Consequently, the weight $W^* \in \mathscr{S}(L^*)$ in the weighted norm~$L^2(W^* d\mu)$ actually diverges at large values of the Hamiltonian~$H$, so the topology induced by~$L^2(W^* d\mu)$ is finer than the one induced by the usual (unweighted) $L^2(d\mu)$ norm.         
\end{Remark}

\subsubsection*{Poincar\'e inequalities.}
Because of the product structure of the invariant measure $\mu$, the $q$-marginal $\mu_{\mathrm{OD}}$ of $\mu$ given by 
\begin{align*}
    \mu_{\mathrm{OD}}(dq)= \frac{1}{\mathscr{N}_{\mathrm{OD}}} \rme^{-U(q)} \, dq, \qquad \mathscr{N}_{\mathrm{OD}} = \int_\mathscr{O} \rme^{-U(q)} \,dq, 
\end{align*}
satisfies a Poincar\'{e} inequality if and only if $\mu$ does, in view of tensorization results for Poincar\'e inequalities. We rely here on the fact that the other term in the product, the $p$-marginal of~$\mu$ denoted by~$\mu_{\mathrm{OU}}$, is a Gaussian measure with identity covariance, so that, for any $\phi\in H^1(d\mu_{\mathrm{OU}})$ with $\int_{\R^d} \phi \, d\mu_{\mathrm{OU}}=0$, it holds
\begin{equation}
  \label{eq:PoincareOU}
  \|\phi\|_{L^2(d\mu_{\mathrm{OU}})}^2\leq \|\nabla_p\phi\|_{L^2(d\mu_{\mathrm{OU}})}^2.
\end{equation}
Thus, in order to analyze constants later, we express the Poincar\'{e} constant for $\mu$ in terms of the Poincar\'{e} inequality for $\mu_{\mathrm{OD}}$. The following estimate is a consequence of the results stated in~\cite{BBCG_08} for instance.

\begin{Proposition}
  \label{prop:poincare}
  Suppose that $U$ satisfies Assumptions~\ref{assump:basic} and~\ref{assump:gc}. Then there exists a constant $\rho>0$ such that the following bound holds for all $\phi \in H^1(d\mu_{\mathrm{OD}})$ with $\int_\X \phi \, d\mu_{\mathrm{OD}} =0$:
  \begin{align*}
   \rho \int_\mathscr{O} \phi^2 \, d\mu_{\mathrm{OD}} \leq  \int_\mathscr{O} |\nabla \phi|^2 \, d\mu_{\mathrm{OD}}.
\end{align*} 
\end{Proposition}

When Assumptions~\ref{assump:basic} and~\ref{assump:gc} hold, $\mu$ therefore also satisfies a Poincar\'{e} inequality: for any $\varphi \in H^1(d\mu)$ with $\int_\X \varphi \, d\mu=0$,
\begin{equation}
  \label{eq:Poincare_mu}
  \min(1,\rho) \|\varphi\|_{L^2(d\mu)}^2 \leq \|\nabla_q\varphi\|_{L^2(d\mu)}^2 + \|\nabla_p\varphi\|_{L^2(d\mu)}^2.
\end{equation}

\subsubsection*{Modified norms.}
For convenience in the arguments in later sections of the paper, we make use of the following notation:
\begin{align}
  \label{eqn:inner}
  \|\phi\|_{T}^2&=\int_\X \big[T\phi\big]\,\phi\, d\mu,
\end{align}
where $T$ is a positive operator. When the argument~$T$ in the above is a nonnegative function, one should understand this as being a multiplication operator. We can further define an inner product~$\langle \cdot, \cdot \rangle_T$ associated to~\eqref{eqn:inner} by polarization. When no argument is indicated in the norm, $\|\phi\|$ denotes the usual norm on~$L^2(d\mu)$.

\subsubsection*{Exponential convergence in $L^2(d\mu)$.}
Our first result establishes exponential convergence to equilibrium in the unweighted setting of $L^2(d\mu)$ when the potential satisfies Assumptions~\ref{assump:basic} and~\ref{assump:gc}.  In particular, the result applies in the setting of a singular Lennard-Jones interaction as in~\cite{HerMat_19}. Note that the Lyapunov structure is not directly employed here.

\begin{Theorem}
  \label{thm:premain}
Suppose that the potential $U$ satisfies Assumptions~\ref{assump:basic} and~\ref{assump:gc}. Then,
\begin{itemize}
\item[(i)] for any $\gamma > 0$, there exist an explicit constant $\lambda >0$ such that, for any $\phi \in L^2(d\mu)$ with $\int_\X \phi \, d\mu =0$, 
\begin{align*}
  \forall t \geq 0, \qquad \|P_t\phi\|\leq 3\exp\left(-\lambda t\right)\|\phi\|.
\end{align*}
\item[(ii)] there exists $\overline{\lambda}$ such that the following lower bounds on the exponential decay rate holds:
\begin{align*}
\forall \gamma \in (0,+\infty), \qquad \lambda \geq \overline{\lambda} \min \{ \gamma, \gamma^{-1} \}. 
\end{align*}
\end{itemize}
\end{Theorem}

\subsubsection*{Exponential convergence in weighted $L^2$ spaces.}
Our next result shows that Assumptions~\ref{assump:basic} and~\ref{assump:gc} imply explicit exponential convergence to equilibrium in a weighted topology constructed from the appropriate Lyapunov functional. The existence of a weak Lyapunov function follows by Remark~\ref{rem:wLF}, but, as remarked earlier, we also have the existence of an explicit strong Lyapunov function $W^*$ under Assumptions~\ref{assump:basic} and~\ref{assump:gc}. In fact, we construct explicit Lyapunov functions, which parametrically depend on the friction~$\gamma>0$, in order to obtain the correct scaling for the convergence rate.

\begin{Theorem}
  \label{thm:main}
  Suppose that the potential $U$ satisfies Assumptions~\ref{assump:basic} and~\ref{assump:gc}.
  \begin{itemize}
  \item[(i)] Consider $W^*\in \mathscr{W}_{\alpha, \beta}(L^*)$ for some constants $\alpha, \beta>0$.  Let $\lambda>0$ be as in the conclusion of Theorem~\ref{thm:premain} and set $m=5\eta\lambda/\beta$ for some~$\eta \in (0,1)$. 
    Then, for any $\phi \in L^2(W^* d\mu)$ with $\int_\X \phi \, d\mu =0$,
    \begin{equation}
      \label{eq:L2weighted}
      \forall t\geq 0, \qquad \|P_t\phi\|_{m W^*+1}\leq 3 \exp\left(-\min\left\{\lambda(1-\eta), \frac\alpha2 \right\} t\right)\|\phi\|_{m W^*+1}.
    \end{equation}
  \item[(ii)] For any $\gamma >0$ and $\eta \in (0,1)$, there exist $\alpha, \beta>0$ and $W^*\in \mathscr{S}_{\alpha, \beta}(L^*)$ of the form
    \[
    W^*= \exp(\eta (H+ \psi_\gamma)),
    \]
    where $\psi_\gamma = \mathrm{o}(H)$ as $H\rightarrow \infty$. Moreover, there exist $c > 0$ and $C,D \in \mathbb{R}_+$ (which all depend on~$\eta$ but not on~$\gamma$) such that 
    \begin{align}
      \forall \gamma \leq 1, & \qquad \alpha \geq c \gamma, \qquad \beta \leq C \gamma, \label{eq:scaling_Thm2_gamma_leq_1} \\
      \forall \gamma \geq 1, & \qquad \alpha \geq c \gamma, \qquad \beta \leq C \gamma^3 \mathscr{M}_\gamma, \qquad \mathscr{M}_\gamma = \max_{|p| \leq D, |\nabla U| \leq D \gamma} W^*. \label{eq:scaling_Thm2_gamma_geq_1}
    \end{align}
  \end{itemize}
\end{Theorem}

Note that the decay rate in~\eqref{eq:L2weighted} is slightly smaller than the one in~\eqref{thm:premain}. Somehow, the larger the Lyapunov function is (\emph{i.e.} the larger~$m$ is), the smaller the decay rate is. Note also that if $W^*\in \mathscr{W}(L^*)\setminus \mathscr{S}(L^*)$, then Theorem~\ref{thm:main} does not provide additional information when compared with Theorem~\ref{thm:premain}. Essentially, one can think of $W^*\in \mathscr{W}(L^*)\setminus \mathscr{S}(L^*)$ as being a positive constant, in which case the norms $\| \cdot \|_{mW^*+1}$ and $\| \cdot \|$ are equivalent. However, if $W^*\in \mathscr{S}(L^*)$, then $W^*$ has compact sub-level sets $\{ W^* \leq c\}$ for all $c>0$.  This means that $W^*\rightarrow \infty$ as $H\rightarrow \infty$, in which case the norm $\| \cdot \|_{mW^* +1}$ dominates the norm~$\|\cdot \|$ (see also Remark~\ref{rmk:finer_topo}).    

Before proceeding further, let us discuss more precisely the convergence result provided by Theorem~\ref{thm:main} through some remarks.  

\begin{Remark}
Item (ii) of Theorem~\ref{thm:main} implies that the effective exponential decay rate~$\min\{\lambda(1-\varepsilon), \alpha/2\}$ in~\eqref{eq:L2weighted} is still of order~$\min\{\gamma, \gamma^{-1}\}$, as in Theorem~\ref{thm:premain}.  When $\gamma >0$ is small, we also note that $m$ is of order~$1$, so that the norm~$\|\cdot\|_{m W^*+1}$ is genuinely stronger than the standard norm on~$L^2(d\mu)$.  On the other hand, when $\gamma > 0$ is large, the factor $m$ is small in~$\gamma$, of order $\gamma^{-4}\mathscr{M}_\gamma^{-1}$. To estimate $\mathscr{M}_\gamma$, one needs more knowledge than Assumptions~\ref{assump:basic} and~\ref{assump:gc} in order to compare $|\nabla U|$ and~$U$. One should however typically think of $\mathscr{M}_\gamma$ as being exponentially large in some power of~$\gamma$ as $\gamma\rightarrow \infty$. The poor scaling of this constant with respect to~$\gamma$ is not surprising given that the construction of the Lyapunov function in the previous result is based on the analysis of the dynamics at large energies, outside the region $\{|p| \leq D\} \cap \{ |\nabla U| \leq D\gamma \}$.  Note that this region grows as $\gamma\rightarrow \infty$ so that the analysis works in a shrinking part of the phase space.           
\end{Remark}

\begin{Remark}
Following the calculations in~\cite{BGH_19}, it is possible to use the methods of this paper to marginally improve the dimensionality dependence in~$\lambda$ of Theorem~\ref{thm:premain}. In particular, it is possible to replace Assumption~\ref{assump:gc} with the condition: There exists $\varepsilon >0$ small enough but independent of the dimension such that condition (7) is satisfied. Indeed, note from the proof of Proposition~2.40 in the appendix of~\cite{BGH_19} that it is possible to choose~$\varepsilon$ in Assumption~\ref{assump:gc} independent of the dimension for Lennard--Jones-like potentials. This can be done by carefully choosing the dimension dependence of the constants~$C_1,C_2$ in the latter work as functions of $d$. However, the resulting constant~$C_\varepsilon$ in~\eqref{assump:gc} will diverge as~$d^a$ (the exponent $a$ being determined by the choices of~$C_1,C_2$). It is possible to remedy to this issue by rescaling variables as $q\mapsto d^{-a/2}q$, so that $|d^{a} \nabla^2 U(q)y|\leq |d^{a/2} \nabla U(q)|^2|y|+C_\epsilon$, which leads to $|\nabla^2 U(q)y|\leq |\nabla U(q)||y|+d^{-a} C_\varepsilon$. Hence, under this spatial rescaling, both $\varepsilon$ and $C_\varepsilon$ can be chosen independently of the dimension.
  
  Although the scaling of~$\varepsilon$ and~$C_\varepsilon$ in Assumption~\ref{assump:gc} can be made precise and controlled, the dependence of the Poincar\'e constant~$\rho$ with respect to the dimension is less clear. Dimension-free Poincar\'e constants are obtained only in unrealistic situations, for instance particles not interacting which each other (in which case the potentials is separable, namely~$U(q)=\sum_{i=1}^d u(q_i)$, and the Poincar\'e constant is simply the minimum of the Poincar\'e constants associated with the potentials~$u_i$), or uniformly convex potentials and perturbations thereof. This makes it difficult to determine the dimension dependence of the convergence rate~$\lambda$, which is based on the constants in~\eqref{eq:Lambda_pm}.
  \end{Remark}

\begin{Remark}
  \label{rmk:avg_high_energy}
Note that $\alpha$ and $\lambda$ have different scalings with respect to~$\gamma>0$ in part~(ii) when $\gamma >0$ is large.  In particular, $\alpha =\mathcal{O}(\gamma)$ while $\lambda = \mathcal{O}(\gamma^{-1})$ as $\gamma \rightarrow \infty$.  To see heuristically why one should expect this discrepancy, note that in the region \begin{align*}
\{ |p| \geq D \} \cup \{|\nabla U| \geq D \gamma \}
\end{align*}
for $D>0$ large, but independent of $\gamma$, the effects of the noise in the system~\eqref{E:StochDiff} can heuristically be considered to be negligible.  Furthermore, suppose that at high energy levels the system~\eqref{E:StochDiff} moves approximately to leading order in time according to the deterministic Hamiltonian dynamics
\begin{equation}
  \label{eqn:detHam}
  \left\{ \begin{aligned}
    \dot{Q}&= P, \\
    \dot{P} &=- \nabla U(Q).  
  \end{aligned} \right.
\end{equation}
This is heuristically the case if, at higher and higher energy levels, the periods of the `orbits' of~\eqref{eqn:detHam} (should such orbits even exist) are short compared to $\gamma^{-1}$.  
Introduce 
\begin{align*}
  \left\langle P^2 \right\rangle_h = \frac{1}{T_h} \int_0^{T_h} |P_s|^2 \, ds,
\end{align*}
where $T_h$ denotes the time spent during one complete Hamiltonian orbit for the dynamics~\eqref{eqn:detHam} on $\{(Q,P)\,: \, H(Q,P) = h\}$.  If we believe for large $h$ that 
\begin{align*}
  \left\langle P^2 \right\rangle_h \approx c h
\end{align*}
for some $c>0$, then we obtain for $H(q,p)=h$ large that
\begin{align*}
\frac{1}{T_h}\E_{(q,p)}\left[ H(q_{T_h}, p_{T_h})\right] &= \frac{1}{T_h}H(q,p) - \frac{\gamma}{T_{h}}\E_{(q,p)}\left[\int_0^{T_h} |p_s|^2  \, ds\right] + \gamma d \\
&\approx \frac{1}{T_h}H(q,p) - \frac{c\gamma}{T_h}\int_0^{T_h} H(q_s, p_s) \, ds + \gamma d. 
\end{align*}
Since intuitively the time $T_h$ is small for $H=h$ large, we expect some Lyapunov-like condition on~$H$, which suggests that $\alpha=\gamma c$ for some $c>0$ independent of $\gamma$ when~$\gamma$ is large.    
\end{Remark}

\begin{Remark}
  We motivate here the scaling $\min\{ \alpha, \lambda\}$ for the convergence rate in~\eqref{eq:L2weighted}. Consider $\eta>0$ and suppose that $W^*\in \mathscr{S}_{(1+\eta) \alpha, \beta}(L^*)$ for some $\alpha, \beta >0$ with $W^* \geq 1$. Then there exists $K^*\subseteq \X$ compact such that  
\begin{align*}
  \forall (q,p) \in (K^*)^c, \qquad \left(L^* W^*\right)(q,p) \leq -\alpha W^*(q,p). 
\end{align*}
In view of~\eqref{eq:Lstar_RLR}, the function~$W$ defined by $W(q,p):=(\mathcal{R}W^*)(q,p) = W^*(q,-p)$ satisfies an analogous estimate for $L$ on the set $K= \{ (q,p) \,: \, (q,-p) \in K^*\}$, namely
\begin{align*}
    \forall (q,p) \in K^c, \qquad (L W)(q,p) \leq - \alpha W(q,p).
\end{align*}
Let $\tau=\inf\{ t\geq 0 \, : \, (q_t, p_t) \in K\}$ be the first time the process $(q_t, p_t)$ solving~\eqref{E:StochDiff} enters~$K$. Since $W^*\geq 1$ and hence $W\geq 1$ as well, the following equality follows from It\^{o}'s formula applied to $\psi(t, q_t,p_t)=\rme^{\alpha t} W(q_t, p_t)$ started in $K^c$ and stopped at time $\tau$:   
\begin{align}
  \label{eqn:laplacet}
  \forall (q,p) \in K^c, \qquad \E_{(q,p)}\left[\exp(\alpha \tau)\right] \leq W(q,p) < \infty.  
\end{align}
One way to interpret the above exponential moment estimate on~$\tau$ is that the speed at which the process $(q_t, p_t)$ returns to $K$ is governed by the parameter $\alpha$. Outside of this compact set, the dynamics returns to $K$ exponentially fast on average as dictated by the value $\alpha>0$ for which the Laplace transform of $\tau$ in~\eqref{eqn:laplacet} is finite. Once the process enters~$K$, mixing occurs according to the local topology given by the norm~$\| \cdot \|$. This suggests in particular the scaling~$\min\{ \alpha, \lambda\}$ in the exponential convergence rate of Theorem~\ref{thm:main}. 
\end{Remark}

\begin{Remark}
  Some convergence results for Fokker--Planck operators associated with Langevin dynamics can be extended to other types of generators, in particular generators associated with piecewise deterministic Markov processes~\cite{dolbeault2015hypocoercivity,andrieu2018hypocoercivity,lu2020explicit}. The generator of the linear Boltzmann dynamics corresponds to replacing the differential opertator~$L_\mathrm{OU}$ in~\eqref{eq:def_L} by the integral operator~$\Pi-1$. An inspection of the proof of Theorem~\ref{thm:premain} shows that this result will still hold, at least formally. Some estimates are unchanged, such as most of the bounds in Sections~\ref{sec:estimates_tech} and~\ref{sec:elliptic_reg}, which involve either functions of the~$q$ variable only or functions in the image of the generator of the Hamiltonian part of the dynamics. Some work would however be required in the proof of Proposition~\ref{prop:key} in Section~\ref{sec:proof_roposition_prop:key} to make all estimates rigorous since hypoellipticity is lost, and some manipulations based on truncations and carr\'e-du-champ formulas would have to be adapted. In contrast, it is not clear whether the results of Theorem~\ref{thm:main} hold. The issue there is to find a Lyapunov function. In the framework of Section~\ref{sub:scen1}, this is done in~\cite[Lemma~3.2]{BRSS17} for a Lyapunov function quadratic in~$p$ (although the scaling with respect to the equivalent of the parameter~$\gamma$ would have to be made explicit). It is not obvious how to extend the approach to Lyapunov functions such as~\eqref{eq:Lyap_proof_Thm2} since our algebraic manipulations in Section~\ref{sub:general} rely on the fact that we work with second order differential operators.
\end{Remark}

\section{Proof of the main general results}
\label{sec:genresults}
We prove in this section the main general results of this paper, namely Theorem~\ref{thm:premain} part~(i) and Theorem~\ref{thm:main} part~(i). These results are established by assuming some technical estimates, whose proofs are postponed to Section~\ref{sec:auxest}. The analysis of the claimed scalings with respect to the friction parameter~$\gamma$, as outlined in the statements of Theorem~\ref{thm:premain} part~(ii) and Theorem~\ref{thm:main} part~(ii), are studied in a second stage, in Section~\ref{sec:scalings}. 

We first motivate and discuss the main ideas behind the proofs in Sections~\ref{sub:theideai} and~\ref{sub:theideaii}. In these two motivating subsections, we make a number of formal manipulations in order to simplify the presentation, without justification, using generic `nice' functions $\phi\in L^2(d\mu)$ with mean zero with respect to $\mu$. The needed manipulations, as well as what is meant by a `nice' $\phi \in L^2(d\mu)$ with mean zero, are made precise in the proofs of the main results in Section~\ref{sec:layout}.      

\subsection{Idea of proof in the unweighted setting (Theorem~\ref{thm:premain})}
\label{sub:theideai}

The first observation we make is that working in the topology $L^2(d\mu)$ is a good start in order to see some elements of a contraction. Indeed, recalling that 
\begin{equation}
  \label{eq:phi_L_phi}
  2\phi L \phi= L(\phi^2) -2\gamma |\nabla_p \phi|^2, 
\end{equation}
we have, by invariance of $\mu$, that
\begin{align}
  \label{eqn:firstcomp}
  \frac{d}{dt}\|P_t\phi\|^2&=2\langle LP_t\phi,P_t\phi\rangle=\langle L(P_t\phi)^2,1\rangle-2\gamma \|\nabla_p P_t\phi\|^2=-2\gamma\|\nabla_p P_t\phi\|^2.
\end{align}
The Poincar\'{e} inequality~\eqref{eq:Poincare_mu} satisfied  by the measure $\mu$ cannot be used at this stage since only $p$-derivatives appear in~\eqref{eqn:firstcomp} through $-2\gamma\|\nabla_p P_t \phi\|$ above. We need to uncover the `missing' $q$-derivatives, namely $-\|\nabla_q P_t\phi\|^2$. This is done by adding a small perturbation to the norm in $L^2(d\mu)$ in order to couple the~$q$ and~$p$ degrees of freedom, and spread the dissipation from~$p$ to~$q$.

This perturbation is encoded in practice by some operator~$A$, following the approach of~\cite{Herau06,dolbeault2015hypocoercivity}. We introduce 
\eqnn{
  \label{eqn:Adef}
  A &= -\big(1+(L_{\mathrm{H}}\Pi)^* (L_{\mathrm{H}}\Pi)\big)^{-1}(L_{\mathrm{H}}\Pi)^* = \left(1-\Pi L_{\mathrm{H}}^2\Pi\right)^{-1}\Pi L_{\mathrm{H}} = \left(1- L_{\mathrm{OD}}\Pi\right)^{-1}\Pi L_\mathrm{H},
}
where $\Pi$ is the projection operator on $L^2(d\mu)$ given by
\eqnn{
  \label{eqn:pidef}
  (\Pi\phi)(q) = \frac{1}{(2\pi )^{d/2}}\int_{\R^d} \phi(q,p) \rme^{-\frac{|p|^2}{2}} \, dp.
}
Effectively, the perturbation $A$ plays a role similar to the twisted gradient as in~\cite{Talay_02, MN06, Villani_2009}, but it has been renormalized to be an operator on $L^2(d\mu)$ as opposed to $H^1(d\mu)$.  Thus, when we uncover the missing $q$-derivatives, they appear in renormalized form.

\begin{Remark}
Strictly speaking, the operator $A$ is defined on a dense domain in $L^2(d\mu)$. Once we prove that it is bounded operator in $L^2(d\mu)$ when restricted to this dense domain, we can extend it to an operator on all of~$L^2(d\mu)$ with the same norm; see Section~\ref{sec:auxest}, in particular Lemma~\ref{lem:standard}.
\end{Remark}

Recalling the notation~\eqref{eqn:inner}, we next consider the modified `norm' $\| \cdot \|_{1+\delta A}$ with $\delta >0$ a parameter to be determined -- chosen in particular so that $\| \cdot \|_{1+\delta A}$ is indeed a norm, equivalent to the standard norm~$\|\cdot\|$. Then, using~\eqref{eqn:firstcomp}, 
\begin{align}
\nonumber \frac{d}{dt}\|P_t\phi\|_{1+\delta A}^2&=\frac{d}{dt} \|P_t \phi\|^2 + \delta \frac{d}{dt}\langle AP_t \phi, P_t \phi \rangle \\ 
\nonumber &= -2\gamma \| \nabla_p P_t \phi \|^2 + \delta\langle A L P_t \phi, P_t \phi\rangle + \delta\langle A P_t \phi, L P_t \phi\rangle\\
\nonumber &= -2\gamma \| \nabla_p P_t \phi \|^2 +\delta\langle A L_{\mathrm{H}} \Pi P_t \phi, P_t \phi \rangle + \delta\langle AL_{\mathrm{H}}(1-\Pi) P_t \phi, P_t \phi\rangle\\
\nonumber &\qquad   + \delta\gamma \langle A L_{\mathrm{OU}} P_t \phi, P_t \phi \rangle + \delta\langle L^* A P_t \phi, P_t \phi \rangle\\
\label{eqn:Tis}&=: -2\gamma \| \nabla_p P_t \phi \|^2 + \delta \left[ T_1(P_t \phi) + T_2(P_t \phi) + T_3(P_t \phi) + T_4(P_t \phi)\right]. 
\end{align}
We now highlight that the term $T_1(P_t \phi)$ provides the missing dissipation in~$q$, while we will see later on, in Section~\ref{sec:layout}, that $T_i(P_t \phi)$ for $i=2,3,4$ are `lower-order' order terms in a sense to be made precise. 

Note first that, for $\phi \in L^2(d\mu)$ with $\int_\X \phi \, d\mu = 0$, one has $\int_\X P_t \phi \, d\mu = \int_{\mathscr{O}} \Pi P_t \phi \, d\mu_{\rm OD} = 0$. We can therefore use the Poincar\'{e} inequality for the overdamped measure $\mu_{\mathrm{OD}}$ (see Proposition~\ref{prop:poincare}) to obtain
\begin{align}
T_1(P_t \phi)= \langle A L_{\mathrm{H}} \Pi P_t \phi, P_t \phi \rangle &= \langle (1- L_{\mathrm{OD}} \Pi)^{-1} L_{\mathrm{OD}} \Pi P_t \phi, \Pi P_t \phi \rangle \nonumber \\
&= - \left\langle (1+ \nabla_q^* \cdot\nabla_q \Pi)^{-1} \nabla_q^*    \cdot \nabla_q   \Pi P_t \phi, \Pi P_t \phi \right\rangle \nonumber \\
&\leq - \frac{\rho}{\rho + 1}\| \Pi P_t \phi\|^2 . \label{eq:coercivity_q}
\end{align}  
Thus, because the operator $\Pi$ is playing the role of a renormalized gradient in $q$, the term on the last line above can be combined with $-2\gamma \| \nabla_p P_t \phi\|^2 \leq -2\gamma \|(1-\Pi)P_t \phi\|^2$, where the last inequality follows from the Poincar\'{e} inequality~\eqref{eq:PoincareOU} in the $p$-marginal $\mu_{\mathrm{OU}}$. This leads to the sought after global dissipation and allows to apply a Gronwall lemma.

For the approach to be effective, we need to carefully choose $\delta>0$ so that (a) $\|\cdot\|_{1+\delta A}$ is actually a norm, equivalent to the standard one on~$L^2(d\mu)$, (b) the remainder terms $T_i(P_t \phi)$ for $i=2,3,4$ can be controlled, and (c) the decay rate has the correct scaling in~$\gamma$ as claimed in Theorem~\ref{thm:premain}(ii). The originality of our results compared to related results for Langevin dynamics~\cite{DKMS13,GS16,IOS19} is that we consider weaker conditions on~$U$ than in previous works, allowing in particular the possibility of potentials with singularities.

\subsection{Idea of the proof for a weighted norm (Theorem~\ref{thm:main})}
\label{sub:theideaii}

Building off the heuristics in the previous section, we now outline some of the differences from the above when we switch to the weighted topology $L^2(W^* d\mu)$ where $W^*\in \mathscr{W}_{\alpha, \beta}(L^*)$ for some $\alpha, \beta>0$.  In fact, if one makes good choices in the construction of the norm and uses the mathematical setup of Section~\ref{sec:setting}, then the result in the weighted setting, in particular Theorem~\ref{thm:main}, readily follows from the previous analysis.  Again, we proceed here formally using `nice' functions with mean zero.    

Recalling the notation~\eqref{eqn:inner}, as well as~\eqref{eq:phi_L_phi}, we obtain for $W^*\in \mathscr{W}_{\alpha, \beta}(L^*)$ that 
\begin{align*}
\frac{d}{dt}\|P_t\phi\|_{W^*}^2&=2\langle LP_t\phi,P_t\phi\rangle_{W^*}=\langle (P_t\phi)^2,L^*W^*\rangle-2\gamma \|\nabla_p P_t\phi\|^2_{W^*}.
\end{align*}
Using the fact that $L^* W^*\leq -\alpha W^* + \beta $, we thus arrive at the inequality 
\begin{align}
\label{eqn:huer1}
\frac{d}{dt}\|P_t\phi\|_{W^*}^2&\leq  - 2\gamma  \| \nabla_p P_t \phi\|^2_{W^*} - \alpha \| P_t \phi \|_{W^*}^2+ \beta \| P_t \phi\|^2.
\end{align}
Somewhat similar to the unweighted setting, we are faced with the problem that, although we have a globally dissipative factor, namely $-\alpha \|P_t \phi \|_{W^*}^2$, we have introduced a growth factor, namely $\beta \|P_t \phi \|^2$.  Thus we need to somehow add a perturbation to the norm that compensates for this term.

In what follows, we perturb the norm $\| \cdot \|_{W^*}$ by $1+\delta A$; that is, we consider $\| \cdot \|_{W^* + 1+ \delta A}$ where $A$ is as in~\eqref{eqn:Adef} and, provided it exists following the analysis in Section~\ref{sub:theideai}, $\delta>0$ satisfies the following two properties:
\begin{itemize}
\item[(i)] $\| \cdot \|_{1+\delta A}$ and $\|\cdot \|$ are equivalent norms;
\item[(ii)] There exists $\lambda >0$ for which the following estimate holds for all $\phi$ in a `nice' dense family in $L^2(d\mu)$ with $\int_\X \phi \, d\mu=0$:
  \begin{align*}
    \forall t \geq 0, \qquad \frac{d}{dt} \|P_t \phi \|_{1+\delta A}^2 \leq - 2\lambda \|P_t \phi \|_{1+\delta A}^2.
  \end{align*}
\end{itemize}
If such a $\delta>0$ exists, we then find that 
\begin{align*}
  \frac{d}{dt} \|P_t \phi \|_{W^*+ 1+ \delta A}^2&= \frac{d}{dt} \| P_t \phi \|_{W^*}^2 + \frac{d}{dt} \| P_t \phi \|_{1+\delta A}^2\\
  & \leq - \alpha \| P_t \phi \|_{W^*}^2 + \beta \| P_t \phi \|^2 - 2\lambda \|P_t \phi \|_{1+\delta A}^2.  
\end{align*}
In order to make sure that $\lambda >0$ is large enough to subsume the $\beta$ term, we rescale the weak Lyapunov function by a factor~$m>0$ to be determined, so that $mW^*\in \mathscr{W}_{\alpha, m \beta}(L^*)$. Then, repeating the above produces
\begin{align*}
  \frac{d}{dt} \|P_t \phi \|_{mW^*+ 1+ \delta A}^2 \leq - \alpha \|P_t \phi \|_{m W^*}^2 + m \beta \| P_t \phi \|^2 - 2\lambda \|P_t \phi \|_{1+\delta A}^2.
\end{align*}
Picking $m>0$ small enough depending on $\lambda, \delta >0$ allows to conclude the argument. 

\subsection{Proof of parts~(i) of Theorems~\ref{thm:premain} and~\ref{thm:main}}
\label{sec:layout}

In order to prove the main general results, we first make precise the class of sufficiently smooth test functions $\phi:\X\rightarrow \R$ we consider. This allows us to more easily manipulate expressions to arrive at the desired inequalities, and then apply density arguments to see that the inequalities are satisfied for a larger class of functions. More precisely, we define 
\[
C_{\mathrm{c},0}^\infty(\X) = \left\{ \phi\in C^\infty(\X) \, \middle| \, \phi = \psi- \int_\X \psi \, d\mu \, \text{ for some } \psi \in C_\mathrm{c}^\infty(\X) \right\}.
\]
We note that a function in $C_{\mathrm{c},0}^\infty(\mathscr{X})$ \emph{does not} have compact support in general, since it is generally constant and non-zero outside of a compact set. We observe that, by Assumption~\ref{assump:basic} and H\"{o}rmander's theorem~\cite{Hor_67}, hypoellipticity implies that the function $P_t \phi$ belongs to~$C_{\mathrm{b}, 0}^\infty(\X)$ for $t>0$ whenever~$\phi \in C_{\mathrm{c},0}^\infty(\X)$, where 
\begin{align*}
  C_{\mathrm{b}, 0}^\infty(\X)= \left\{ \phi \in C^\infty(\X)\, \middle| \, \sup_{x \in \X}|\phi(x)| < +\infty, \,\, \int_\X \phi \, d\mu =0 \right\}.
\end{align*}

In order to conclude Theorem~\ref{thm:premain}(i) and Theorem~\ref{thm:main}(i), we assume that Proposition~\ref{prop:key} below holds true. This result provides some integrated form of dissipation. The proof of this proposition is postponed to Section~\ref{sec:auxest} since it relies on a number of technical estimates and commutations of operators. The approach for obtaining it, however, is exactly the same as in~\cite{dolbeault2015hypocoercivity}, but there are some key differences in our setting, especially in the proof of the `elliptic regularity estimate', which make the estimates slightly more involved as we are allowing a weaker growth condition on~$U$ than in~\cite{dolbeault2015hypocoercivity}.  

In order to state the proposition, observe that Assumption~\ref{assump:gc} implies that there exist constants $c_1\in(0,1)$ and $C_2>0$ independent of $\gamma$ such that
\begin{align}
\label{eqn:c1C2}
\Delta U(q)\leq c_1|\nabla U(q)|^2+C_2.
\end{align}
Consider $\varepsilon >0$ such that 
\begin{align}
  \label{eqn:eps}
  0 < \varepsilon < \frac{(1-c_1)^2}{4},
\end{align}
and let $C_\varepsilon$ be the corresponding constant as provided by Assumption~\ref{assump:gc}. Finally, define the constant $\eta_\varepsilon>0$ as
\begin{align}
  \label{eqn:constants}
  \eta_\varepsilon = \sqrt{2\left(1-\frac{4\varepsilon}{(1-c_1)^2}\right)^{-1}\max\left\{1,\frac{1}{2}\left(C_\varepsilon+\frac{2C_2\varepsilon}{1-c_1}\right)\right\}}. 
\end{align}

\begin{Proposition}
  \label{prop:key}
  Suppose that $U$ satisfies Assumptions~\ref{assump:basic} and~\ref{assump:gc}, and consider $W^*\in \mathscr{W}_{\alpha, \beta}(L^*)$ for some $\alpha, \beta >0$. Then, the following properties hold.  
  \begin{itemize}
  \item[(i)] The operator $A$ is bounded, and satisfies the estimate:
  \begin{align*}
    \forall \phi \in L^2(d\mu), \qquad \| A \phi \| \leq \frac{1}{2}\| (1-\Pi)\phi \|.
  \end{align*}
  \item[(ii)] For any $\phi \in C_{\mathrm{c},0}^\infty(\X)$ and $0 \leq s \leq t$, 
    \begin{align*}
      &\langle A P_t \phi, P_t \phi \rangle -  \langle AP_s  \phi, P_s\phi \rangle+ \int_s^t  \frac{ \rho}{\rho+1}\| \Pi P_u \phi\|^2 \, du\\
      &\qquad \qquad \leq  \int_s^t  
      \left[\left(\eta_\varepsilon+\frac\gamma2\right) \| (1-\Pi) P_u \phi\| \| \Pi P_u \phi\|  + \| (1-\Pi) P_u \phi\|^2      \right] du ,
    \end{align*}
    where the constant $\eta_\varepsilon >0$ is defined in~\eqref{eqn:constants}.
  \item[(iii)] For any $\phi \in C^\infty_{\mathrm{c},0}(\X)$ and $0 \leq s \leq t$, 
    \begin{align*}
      \|P_t \phi \|_{W^*}^2 \leq \| P_s\phi\|_{W^*}^2 - \int_s^t \left[ \alpha \| P_u \phi \|_{W^*}^2 - \beta \|P_u \phi \|^2 + 2\gamma \| \nabla_p(P_u \phi) \|^2_{W^*} \right] du. 
    \end{align*} 
  \end{itemize}
\end{Proposition}

Note that, if $W^*\equiv 1$, it is possible to take $\alpha=\beta$ in the estimate of part~(iii), which then corresponds to the integral version of the simple decay estimate~\eqref{eqn:firstcomp}.   

With Proposition~\ref{prop:key} at hand, we can now complete the proofs of Theorem~\ref{thm:premain}(i) and Theorem~\ref{thm:main}(i).  

\begin{proof}[Proof of Theorem~\ref{thm:premain}(i)]
Let $\phi \in C_{\mathrm{c},0}^\infty(\X)$ and $\delta>0$.  Applying Proposition~\ref{prop:key}(ii) and~(iii) (with $W^*\equiv 1$ and $\alpha =\beta$ in part~(iii)), we have, for $0 \leq s \leq t$,
\begin{align*}
\| P_t \phi \|_{1+\delta A}^2 &\leq \| P_s \phi \|_{1+\delta A}^2 - \left(\int_s^t 2\gamma \| \nabla_p (P_u \phi) \|^2 + \frac{\delta \rho}{\rho+1}\| \Pi P_u \phi\|^2 \, du \right)\\
& \qquad + \delta \int_s^t \left[ \left(\eta_\varepsilon+\frac\gamma2\right) \| (1-\Pi) P_u \phi\| \| \Pi P_u \phi\| + \| (1-\Pi) P_u \phi\|^2 \right] du\\
&=: \| P_s \phi \|_{1+\delta A}^2 - \int_s^t S_1(P_u \phi) \, du + \delta \int_0^t S_2(P_u \phi) \, du.  
\end{align*}
Next, observe that, for any $\phi \in C_{\mathrm{b},0}^\infty(\mathscr{X})$, the Poincar\'{e} inequality~\eqref{eq:PoincareOU} gives 
\begin{align*}
  S_1(\phi) &= 2\gamma\|\nabla_p\phi\|^2 + \frac{\delta\rho}{\rho + 1}\| \Pi \phi\|^2 \geq 2\gamma\|(1-\Pi)\phi\|^2+\frac{\delta\rho}{\rho+1}\|\Pi \phi\|^2.
\end{align*}  
Combining this estimate with the other term~$S_2(\phi)$ produces
\begin{align}
  \label{eqn:quadform}
  S_1(\phi)-\delta S_2(\phi) &\geq X^T \mathbf{B} X,
\end{align}
with 
\begin{align*}
X=\begin{pmatrix}\|(1-\Pi)\phi\|\\
\|\Pi\phi\|\end{pmatrix},
\qquad
\mathbf{B}=
\begin{pmatrix}
  2\gamma-\delta & \dps -\frac{\delta}{2}\left(\eta_\varepsilon+\frac{\gamma}{2}\right)\\
  \dps -\frac{\delta}{2}\left(\eta_\varepsilon+\frac{\gamma}{2}\right) & \dps \frac{\delta\rho}{1+\rho}
\end{pmatrix}.
\end{align*}
Therefore, 
\begin{align*}
   S_1(\phi)-S_2(\phi) \geq \min\{ \Lambda_{+}, \Lambda_{-}\} \|\phi\|^2,
\end{align*}
where $\Lambda_+$, $\Lambda_{-}$ are the eigenvalues of the symmetric matrix $\mathbf{B}$. To analyze the eigenvalues, we denote the trace and determinant of $\mathbf{B}$ by
\begin{align}
  \label{eqn:eigenvalues}
  \mathbf{T}:=\mathrm{Tr}(\mathbf{B})=2\gamma-\frac{\delta}{1+\rho},
  \qquad
  \mathbf{D}:=\det(\mathbf{B})=\frac{\delta\rho}{1+\rho}\left(2\gamma-\delta\right)-\frac{\delta^2}{4}\left(\eta_\varepsilon+\frac{\gamma}{2}\right)^2.
\end{align} 
We can then express the eigenvalues of $\mathbf{B}$ as
\begin{equation}
  \label{eq:Lambda_pm}
  \Lambda_\pm=\frac{\mathbf{T}}{2}\pm \sqrt{\frac{\mathbf{T}^2}{4}-\mathbf{D}}=\frac{\mathbf{D}}{\frac{\mathbf{T}}{2}\mp\sqrt{\frac{\mathbf{T}^2}{4}-\mathbf{D}}},
\end{equation}
provided the denominator in the last equality is non-zero.  We wish to choose $\delta>0$ such that both eigenvalues are indeed positive. Since~$\mathbf{B}_{22}>0$, this is the case when~$\mathbf{D}>0$, \emph{i.e.} 
\eqnn{
  \label{E:Dpos}
  \delta<\frac{2\gamma}{1+\frac{1+\rho}{4\rho}\left(\eta_\varepsilon+\frac{\gamma}{2}\right)^2}.
}

On the other hand, we need to make sure we are picking $\delta >0$ so that $\| \cdot \|$ and $\| \cdot \|_{1+\delta A}$ are equivalent norms.  Proposition~\ref{prop:key}(i) implies that, for $0<\delta <2$,
\begin{equation}
  \label{eq:equivalence_norms}
  \left(1-\frac{\delta}{2}\right)\|\phi\|\leq\|\phi\|_{1+\delta A}\leq\left(1+\frac{\delta}{2}\right)\|\phi\|.
\end{equation}
Therefore, for any
\begin{align}
  \label{eqn:deltarestriction}
  \delta<\min\left\{1,\frac{2\gamma}{1+\frac{1+\rho}{4\rho}\left(\eta_\varepsilon+\frac{\gamma}{2}\right)^2}\right\},
\end{align}
it holds that $\|\phi\|_{1+\delta A}\leq \frac32 \|\phi\|$ and
\begin{align}\label{eqn:preGronwall}
    \|P_t\phi\|_{1+\delta A}^2\leq \|P_s \phi\|_{1+\delta A}^2 -5 \lambda  \int_s^t \|P_u\phi\|^2 \, du \leq  \|P_s \phi\|_{1+\delta A}^2 - 2\lambda\int_s^t \| P_u \phi \|^2_{1+\delta A}\, du,  
\end{align}
where
\[
\lambda=\frac{1}{5}\min\{\Lambda_+, \Lambda_{-}\} = \frac{\Lambda_-}{5}.
\]


The inequality~\eqref{eqn:preGronwall} implies by a Gronwall lemma that $\|P_t\phi\|_{1+\delta A}^2\leq \exp\left(-2\lambda t\right)\|\phi\|_{1+\delta A}^2$. Hence, by the equivalence of norms~\eqref{eq:equivalence_norms} and the choice of $\delta$, we have
\begin{align}
    \|P_t\phi\|^2 \leq 9 \exp(-2\lambda t) \| \phi \|^2.  
\end{align}
Since the above inequality is satisfied for any~$\phi \in C_{\mathrm{c},0}^\infty(\X)$, it follows by density that it is also satisfied for any $\phi \in L^2(d\mu)$ with $\int_\X \phi \, d\mu =0$, which allows to conclude the proof.
\end{proof}
  
With the proof of Theorem~\ref{thm:premain}(i) at hand, we can next turn to the proof of Theorem~\ref{thm:main}(i).

\begin{proof}[Proof of Theorem~\ref{thm:main}(i)]
Suppose that $W^*\in \mathscr{W}_{\alpha, \beta}(L^*)$ for some $\alpha, \beta >0$ and let $\delta >0$ be as in~\eqref{eqn:deltarestriction}.  We first observe that $mW^* \in \mathscr{W}_{\alpha, m \beta}(L^*)$ for any $m>0$. Applying Proposition~\ref{prop:key}(iii) and~\eqref{eqn:preGronwall}, we have, for any $\phi \in C^\infty_{\mathrm{c},0}(\X)$,
\begin{align}
   \nonumber \| P_t \phi \|_{m W^*+ 1+\delta A}^2 &=\| P_t \phi \|^2_{m W^*} + \|P_t \phi \|_{1+ \delta A}^2\\
    \label{eqn:mW} & \leq  
   \| P_s \phi \|_{mW^*+1+\delta A}^2  - \int_s^t \left[ \alpha  \| P_u \phi \|_{m W^*}^2 - m \beta \|P_u \phi \|^2 + 5 \lambda \|P_u \phi \|^2 \right] du,  
\end{align}
where $\lambda$ is as in~\eqref{eqn:deltarestriction}. Observe that, by choosing
\begin{align}
  m=\frac{5\eta \lambda}{\beta}
\end{align}
for some~$\eta \in (0,1)$, we find that 
\begin{align*}
-\alpha  \| P_u \phi \|_{m W^*}^2+ m \beta \|P_u \phi \|^2 - 5\lambda \|P_u \phi \|^2 & \leq - \min \{ \alpha, 2\lambda(1-\eta) \} \| P_u \phi \|_{mW^* +1 +\delta A}^2,
\end{align*}
where we used the equivalence of norms~\eqref{eq:equivalence_norms} and the choice of $\delta \in [0,1]$ as in~\eqref{eqn:deltarestriction}. In particular, it follows from a Gronwall lemma, an approximation argument and the equivalence of norms~\eqref{eq:equivalence_norms}, that, for any $\phi \in L^2( W^* \, d\mu)$ with $\int_{\X} \phi \, d\mu =0$,
\[
\forall t \geq 0, \qquad \|P_t \phi \|^2_{m W^* +1} \leq 9 \exp(-\min\{2\lambda(1-\varepsilon), \alpha\} t)\| \phi \|_{mW^* +1}^2.
\]
This gives the claimed result.
\end{proof}

\section{Scalings and consequences of Lyapunov structure}
\label{sec:scalings}

In this section, we analyze the rate of convergence to equilibrium in both the unweighted and weighted settings with respect to the friction parameter $\gamma>0$, ultimately proving Theorem~\ref{thm:premain}(ii) and Theorem~\ref{thm:main}(ii). In the unweighted setting (Section~\ref{sec:unweighted_lambda}), we study in particular the behavior the parameter $\lambda$ as $\gamma\rightarrow 0$ or $\gamma\rightarrow \infty$. In the weighted setting (Section~\ref{sec:Lyapunov}), we need of course the previous analysis of $\lambda$, but we also have to construct a Lyapunov function $W^*\in \mathscr{S}_{\alpha, \beta}(L^*)$ for some $\alpha, \beta >0$, with an explicit dependence of~$\alpha,\beta$ on~$\gamma>0$.     

\subsection{The unweighted setting and $\lambda$}
\label{sec:unweighted_lambda}

Here we recall that 
\begin{align*}
\lambda = \frac{1}{5}\max\{ \Lambda_+, \Lambda_- \}= \frac{1}{5}\Lambda_-,
\end{align*}
 where $\Lambda_{\pm}$ are the eigenvalues introduced in~\eqref{eq:Lambda_pm}. We also recall that $\delta >0$ needs to be chosen so that the restriction~\eqref{eqn:deltarestriction} is satisfied, and that $\eta_\varepsilon, \rho>0$ are constants which do not depend on~$\gamma$. 

\begin{proof}[Proof of Theorem~\ref{thm:premain}(ii)]
  We briefly recall the approach of~\cite[Section~5.3]{IOS19} for instance. Consider first the case when $\gamma \geq 1$. We choose $\delta = c \gamma^{-1}$ with $c>0$ small enough (independent of~$\gamma$) so that~\eqref{eqn:deltarestriction} is satisfied. Then, by simplifying the expression of $\Lambda_-$, we find that $\mathbf{T} \sim 2\gamma$ and~$\mathbf{D} \sim \frac{2c\rho}{1+\rho}-\frac{c^2}{16}$ as $\gamma\rightarrow \infty$, so that
 \begin{align*}
 \Lambda_- = \frac{c}{\gamma} \left( \frac{\rho}{\rho+1} - \frac{c}{32}\right) + \mathcal{O}\left(\frac{1}{\gamma^2}\right).  
 \end{align*}  
The prefactor of the dominant term~$\gamma^{-1}$ is positive for~$c>0$ sufficiently small, hence giving the correct scaling for~$\lambda$ as~$\gamma \to \infty$.  

We next consider the situation when $\gamma \leq 1$, for which we choose $\delta = c \gamma$ with~$c>0$ sufficiently small so that~\eqref{eqn:deltarestriction} holds. A simple computation shows that, as $\gamma \to 0$,
\begin{align*}
  \Lambda_- = \gamma \left[ 1 - \frac{c}{2(1+\rho)} - \sqrt{\left(1 - \frac{c}{2(1+\rho)}\right)^2 - \frac{c \rho}{1+\rho} (2-c) +\frac{c^2}{4} \eta_\varepsilon^2}\right] + \mathcal{O}(\gamma^2).  
\end{align*}
The prefactor of the dominant term~$\gamma$, equal to~$c\rho/(1+\rho) + \mathrm{O}(c^2)$, is positive for~$c>0$ sufficiently small, hence giving the correct scaling for~$\lambda$ as~$\gamma\to0$. This concludes the proof of the result. 
\end{proof}
 
\subsection{Explicit Lyapunov functions and convergence rates}
\label{sec:Lyapunov}

Our goal in this section is to construct an explicit Lyapunov function $W^*\in \mathscr{S}_{\alpha, \beta}(L^*)$ and analyze the value $\alpha >0$ we obtain from the construction.  Since the convergence rate parameter in the weighted setting is $\min\{ \alpha/2, \lambda(1-\varepsilon)\}$ where $\lambda>0$ is the parameter analyzed in Section~\ref{sec:unweighted_lambda} and~$\varepsilon \in (0,1)$, our ultimate goal here is to check whether the parameter $\alpha >0$ also scales as $\min(\gamma,\gamma^{-1})$ as $\gamma \to 0$ or $\gamma \to +\infty$. 

As a simple first example, we consider in Section~\ref{sub:scen1} a commonly employed condition on~$U$ in the literature (see~\cite{Wu_01, Talay_02, MSH_02}), relevant for potentials which have a polynomial-like growth at infinity. In this setting, we follow various previous works~\cite{Wu_01, Talay_02, MSH_02} and construct our Lyapunov function by adding a term~$p\cdot q$ to the Hamiltonian. We make sure that the scaling of~$\alpha$ with respect to~$\gamma$ is indeed~$\min(\gamma,\gamma^{-1})$. We then consider in Section~\ref{sub:general} the general situation of a potential satisfying Assumptions~\ref{assump:basic} and~\ref{assump:gc}, for which we construct a different Lyapunov function allowing us to write the proof of Theorem~\ref{thm:main}(ii).   

\subsubsection{Polynomial-like potentials}
\label{sub:scen1}

We suppose here that the potential $U$ satisfies Assumptions~\ref{assump:basic} and~\ref{assump:gc} and the following additional growth assumption.

\begin{Assumption}
  \label{cond:3}
  The potential $U$ is such that $U\in C^\infty(\R^d)$ and there exist constants $c_3, C_4, c_5 >0$ for which
  \[
  \forall q\in\R^d, \qquad \nabla U(q)\cdot q\geq c_3 \,U(q)-C_4, \qquad U(q) \geq c_5|q|^2.
  \]
\end{Assumption}

Under these assumptions, we now build a Lyapunov function $W\in \mathscr{S}_{\alpha, \beta}(L)$ in the same spirit as in \cite{Wu_01, Talay_02, MSH_02}.

\begin{Lemma}
  \label{lem:Polynomial-Like}
  Suppose that $U$ satisfies Assumptions~\ref{assump:basic}, \ref{assump:gc} and~\ref{cond:3}, and introduce
  \begin{align}
    \label{eq:W_eta_kappa}
    W_{\kappa}(q,p) = H(q,p)+\kappa\,q\cdot p.
  \end{align}
  Then there exist~$\overline{\kappa}>0$ such that
  \eqn{
    \forall\gamma\in(0,1],\qquad LW_{\overline{\kappa}\gamma}(q,p)&\leq -\frac{\overline{\kappa}c_3}{2}\gamma \, W_{\overline{\kappa}\gamma} + \gamma d+C_4\overline{\kappa}\gamma, \\
    \forall\gamma\in[1,\infty),\qquad LW_{\overline{\kappa}/\gamma}(q,p)&\leq-\frac{\overline{\kappa}c_3}{2\gamma} W_{\overline{\kappa}/\gamma}+ \gamma d+\frac{C_4\overline{\kappa}}{\gamma}.
    }
\end{Lemma}

Note that the scaling we obtain on the decay rate~$\alpha$ for~$\gamma \geq 1$ is not of order~$\gamma$ as in Theorem~\ref{thm:main}(ii), but of order~$1/\gamma$. This is not an issue since the associated exponential decay rate in~\eqref{eq:L2weighted} will still scale as~$\min(\gamma,\gamma^{-1})$. The difference with the scaling obtained in Theorem~\ref{thm:main}(ii) comes from the fact that the Lyapunov function is not of exponential type here. The smaller value of~$\alpha$ is compensated by a much smaller and more explicit value of~$\beta$.

\begin{proof}
  In order for~$W_{\kappa}>0$ and $W_{\kappa}\rightarrow \infty$ as $H\rightarrow \infty$, it is sufficient by Assumption~\ref{cond:3} and a Cauchy--Schwarz inequality that 
\begin{align}
  \label{eqn:firstbound}
  \kappa < \sqrt{2 c_5}. 
\end{align}   
In view of Assumptions~\ref{assump:basic} and~\ref{cond:3}, 
\begin{align} 
L W_{\kappa} (q,p)&= -(\gamma-\kappa)|p|^2-\kappa \nabla U(q)\cdot q-\kappa \gamma \,p\cdot q+\gamma d \nonumber  \\
&\leq -(\gamma-\kappa)|p|^2-c_3\kappa \,U(q)-\kappa \gamma \,p\cdot q +\gamma d+C_4\kappa. \label{eq:LW_spec}
\end{align}
When $\delta \leq c_3 \kappa$, it holds
\begin{align*}
(\gamma-\kappa)|p|^2+c_3\kappa \,U(q)+\kappa \gamma \,p\cdot q-\delta W_{\kappa}(q,p) \geq \begin{pmatrix} q \\ p \end{pmatrix}^T \begin{pmatrix} (c_3 \kappa-\delta)c_5 & \kappa(\gamma-\delta)/2 \\ \kappa(\gamma-\delta)/2 & \gamma-\kappa-\delta/2 \end{pmatrix}\begin{pmatrix} q \\ p \end{pmatrix}.  
\end{align*}
The aim is to prove that the matrix appearing on the right-hand side of the previous inequality is nonnegative.

We first consider the case $\gamma \leq 1$. Setting $\kappa=\overline{\kappa}\gamma$ and $\delta = c_3\kappa/2$, we obtain 
\[
\begin{pmatrix} (c_3 \kappa-\delta)c_5 & \kappa(\gamma-\delta)/2 \\ \kappa(\gamma-\delta)/2 & \gamma-\kappa-\delta/2 \end{pmatrix} = \gamma \begin{pmatrix} \dps \frac{c_3 c_5}{2} \overline{\kappa} & \dps \frac{\overline{\kappa}}{2} \gamma \left(1-\frac{c_3\overline{\kappa}}{2}\right) \\ \dps \frac{\overline{\kappa}}{2} \gamma \left(1-\frac{c_3\overline{\kappa}}{2}\right) & \dps 1-\overline{\kappa}\left(1+\frac{c_3}{4}\right) \end{pmatrix}.
\]
We finally choose $\overline{\kappa}>0$ sufficiently small so that the latter matrix is positive for all~$\gamma \leq 1$ (which is possible since the determinant is of order~$c_3 c_5 \overline{\kappa}/2 + \mathrm{O}(\overline{\kappa}^2)$ uniformly in~$\gamma \leq 1$) and~\eqref{eqn:firstbound} is satisfied for $\gamma = 1$, which provides the claimed inequality for~$\gamma \leq 1$. 

For $\gamma \geq 1$, we still set $\delta = c_3\kappa/2$ but consider now $\kappa = \overline{\kappa}/\gamma$. Then,
\[
\begin{pmatrix} (c_3 \kappa-\delta)c_5 & \kappa(\gamma-\delta)/2 \\ \kappa(\gamma-\delta)/2 & \gamma-\kappa-\delta/2 \end{pmatrix} = \begin{pmatrix} \dps \frac{c_3 c_5 \overline{\kappa}}{2\gamma} & \dps \frac{\overline{\kappa}}{2} \left(1-\frac{c_3\overline{\kappa}}{2\gamma^2}\right) \\ \dps \frac{\overline{\kappa}}{2} \left(1-\frac{c_3\overline{\kappa}}{2\gamma^2}\right) & \dps \gamma-\frac{\overline{\kappa}}{\gamma}\left(1+\frac{c_3}{4}\right) \end{pmatrix}.
\]
When~$\overline{\kappa}>0$ is sufficiently small, the determinant of the matrix on the right-hand side is positive, of order~$c_3 c_5 \overline{\kappa}/2 + \mathrm{O}(\overline{\kappa}^2)$ uniformly in~$\gamma \geq 1$. Upon further reducing the value~$\overline{\kappa}$ found for~$\gamma \leq 1$, we finally obtain the claimed inequalities.
\end{proof}

\subsubsection{The general case} 
\label{sub:general}

If we remove Assumption~\ref{cond:3}, then this limits the types of known Lyapunov functions one can consider. In fact, the function in~\eqref{eq:W_eta_kappa} will not satisfy~\eqref{eq:LW_spec} as a Lyapunov function if we merely consider Assumptions~\ref{assump:basic} and~\ref{assump:gc}. In order to deal with these issues, we will slightly modify the form of the function as in~\cite{HerMat_19, BGH_19}. Note that that a similar form was used in~\cite{CEHRB_18,LuMat_19}.  

\begin{proof}[Proof of Theorem~\ref{thm:main}(ii)]
Our goal here is to construct a Lyapunov function for~$L$ of the form
\begin{equation}
  \label{eq:Lyap_proof_Thm2}
  W(q,p) = \rme^{\eta \widetilde{H}(q,p)},
\end{equation}
where 
\[
\widetilde{H}(q,p) = H(q,p)+\kappa\frac{p\cdot\nabla U(q)}{|\nabla U(q)|^2+\sigma}=: H(q,p) + \kappa \psi(q,p),
\]
for some constants $\kappa, \sigma>0$ and~$\eta \in (0,1)$ to be determined later. The Lyapunov function for~$L^*$ is then obtained by Proposition~\ref{prop:relLyap}.

First observe that $|\psi(q,p)| \leq \sigma^{-1/2}|p|$, so that, for any choice of $\kappa, \sigma, \eta >0$, it holds $W>0$ and $W\to+\infty$ as $H\to+\infty$. Furthermore, for any $\eta \in (0,1)$ and $\kappa, \sigma >0$, it is easy to check that $W$ is strongly integrable (see Definition~\ref{def:Lyap}). Next, to help compute $L W$, note that 
\eqn{
  L \widetilde{H}(q,p)= LH + \kappa L \psi &=-\gamma|p|^2 + \gamma d-\kappa\gamma\frac{p\cdot\nabla U(q)}{|\nabla U(q)|^2+\sigma}-\kappa\frac{|\nabla U(q)|^2}{|\nabla U(q)|^2+\sigma}\\
  &\qquad+\kappa\frac{p\cdot\nabla^2 U(q) p}{|\nabla U(q)|^2+\sigma}-2\kappa\frac{p\cdot\nabla U(q)}{|\nabla U(q)|^2+\sigma}\,\frac{\nabla U(q)\cdot\nabla^2 U(q) p}{|\nabla U(q)|^2+\sigma}.
}
Using Assumption~\ref{assump:gc} on the two terms of the second line gives
\eqn{L\widetilde{H}(q,p)
&\leq\gamma d+\left[-\gamma+\kappa\frac{\varepsilon|\nabla U(q)|^2+C_\varepsilon}{|\nabla U(q)|^2+\sigma}\left(1+2\frac{|\nabla U(q)|^2}{|\nabla U(q)|^2+\sigma}\right)\right]|p|^2-\kappa\frac{|\nabla U(q)|^2}{|\nabla U(q)|^2+\sigma}\\
&\qquad-\kappa\gamma\frac{p\cdot\nabla U(q)}{|\nabla U(q)|^2+\sigma}.}
We choose $\sigma > C_\varepsilon/\varepsilon$, so that the coefficient of~$|p|^2$ achieves a maximum as $|\nabla U(q)| \to +\infty$; the corresponding value being smaller than~$-\gamma+3\kappa \varepsilon$. Therefore, 
\eqn{L\widetilde{H}(q,p)&\leq\gamma d+\left(-\gamma+3\kappa\varepsilon\right)|p|^2-\kappa\frac{|\nabla U(q)|^2}{|\nabla U(q)|^2+\sigma}-\kappa\gamma\frac{p\cdot\nabla U(q)}{|\nabla U(q)|^2+\sigma}.}
Using this inequality along with the fact that for all $V\in C^2(\mathscr{O})$ and $\eta>0$,
\[
L\left( \rme^{\eta V} \right) = \eta \left( LV+\eta\gamma|\nabla_p V|^2\right)\rme^{\eta V},
\]
the function $W = \rme^{\eta \widetilde{H}}$ in~\eqref{eq:Lyap_proof_Thm2} satisfies
\begin{align*}
\frac{L W(q,p)}{\eta W(q,p)}&= L\widetilde{H}(q,p)+\eta\gamma \left|\nabla_p \widetilde{H}(q,p)\right|^2\\
&\leq \gamma d+\left(-\gamma+3\kappa\varepsilon\right)|p|^2-\kappa\frac{|\nabla U(q)|^2}{|\nabla U(q)|^2+\sigma}-\kappa\gamma\frac{p\cdot\nabla U(q)}{|\nabla U(q)|^2+\sigma}\\
&\qquad+\eta\gamma \left(|p|^2+2\kappa\frac{p\cdot\nabla U(q)}{|\nabla U(q)|^2+\sigma}+\kappa^2 \frac{|\nabla U(q)|^2}{(|\nabla U(q)|^2+\sigma)^2}\right)\\
&\leq \gamma d  - \left((1-\eta) \gamma - 3 \kappa \varepsilon\right) |p|^2 - \kappa\frac{|\nabla U(q)|^2}{|\nabla U(q)|^2+ \sigma} + \frac{\kappa^2 \gamma}{|\nabla U(q)|^2 + \sigma} \\
& \qquad + |2\eta-1|\kappa\gamma\frac{\left| p\cdot\nabla U(q) \right|}{|\nabla U(q)|^2+\sigma},
\end{align*}
where we used that $\eta \in [0,1]$ (so that $|2\eta-1| \leq 1$ in particular). Using Young's inequality for the last term, hereby introducing a constant~$C_\eta >0$ which can be made as large as wanted depending on the value of~$\eta$, we finally obtain
\begin{equation}
  \label{eq:LW_on_W}
  \frac{L W(q,p)}{\eta W(q,p)} \leq \gamma d  - \left[\left(1-\eta-C_\eta^{-1}\right) \gamma - 3 \kappa \varepsilon \right] |p|^2 - \kappa\frac{|\nabla U(q)|^2}{|\nabla U(q)|^2+ \sigma} + \left(\frac{C_\eta}{4}+1\right)  \frac{\kappa^2 \gamma}{|\nabla U(q)|^2 + \sigma}.
\end{equation}

We next choose the parameters for the function~$W$ in~\eqref{eq:Lyap_proof_Thm2}. These parameters are the same for the two limiting regimes we consider, namely $\gamma\ll 1$ and $\gamma \gg 1$, but the estimate for $\beta$ changes depending on which case is considered. Let us first fix some value $\eta \in (0,1)$, independently of $\gamma$. The regime where $U$ is large but $|p|$ is small forces us to choose~$\kappa$ to be order~$\gamma$. This is because, in particular, when $|p|$ is small and $U(q)$ is sufficiently large, we need 
\begin{align*}
\gamma d- \kappa\frac{|\nabla U(q)|^2}{|\nabla U(q)|^2+ \sigma} < 0. 
\end{align*}
Specifically, we set $\kappa = 2\gamma d$ and consider $\varepsilon, C_\eta>0$ independent of $\gamma$ and such that 
\begin{align*}
6d \varepsilon +C_\eta^{-1} \leq \frac{1-\eta}{2},  
\end{align*}
so that $(1-\eta-C_\eta^{-1}) \gamma - 3 \kappa \varepsilon \geq (1-\eta)\gamma/2$.

With these choices, consider first the case $\gamma \leq 1$.  Then, on a region of the form
\begin{align*}
\mathscr{R}= \{ |p| > D\} \cup \{ |\nabla U| > D\},
\end{align*}
where $D>0$ is large but independent of $\gamma$, we have that 
\begin{align*}
L W \leq -c \gamma W
\end{align*}
where $c>0$ is independent of $\gamma$.  Now the complementary region 
\begin{align*}
\mathscr{R}^c = \{ |p| \leq D \} \cap \{ |\nabla U| \leq D\}
\end{align*} 
is compact. Moreover, $W$ is in this case bounded on $\mathscr{R}^c$ independently of $\gamma \leq 1$ since $0 \leq \kappa \leq 2d$. Since the right-hand side~\eqref{eq:LW_on_W} is of order~$\gamma$ on~$\mathscr{R}$, this in turn implies the estimate
\begin{align*}
L W \leq -c \gamma W + C \gamma
\end{align*} 
for some constant $C>0$ independent of $\gamma$, which indeed gives~\eqref{eq:scaling_Thm2_gamma_leq_1}. Note that this agrees with the estimate in Lemma~\ref{lem:Polynomial-Like} obtained when~$U$ satisfies Assumption~\ref{cond:3} in addition to Assumption~\ref{assump:gc}, but the function~$W$ constructed here grows much faster for~$H$ large than the one in~\eqref{eq:W_eta_kappa}.  

We next consider the case $\gamma \geq 1$.  Again, we maintain the same choices of $\kappa, \varepsilon, C_\eta$ as before, but this time we need to enlarge the region depending on the size of $\gamma$ in order to control the remainder term $(1+C_\eta/4) \kappa^2 \gamma/(|\nabla U|^2 + \sigma)$. More precisely, in the region
\begin{align*}
\mathscr{R}_\gamma= \{|p| > D \}\cup \{ |\nabla U | > D \gamma\}
\end{align*}
where $D>0$ is large but independent of $\gamma$, it holds 
\begin{align*}
  \forall (q,p) \in \mathscr{R}_\gamma, \qquad (L W)(q,p) \leq - c \gamma W(q,p)
\end{align*}
for some $c>0$ independent of~$\gamma$. On the complementary region 
\begin{align*}
\mathscr{R}^c_\gamma= \{ |p| \leq D \} \cap \{ |\nabla U | \leq D \gamma \},
\end{align*}
the function~$W$ is bounded by~$\mathscr{M}_\gamma:=\max_{\mathscr{R}_\gamma^c} W$. We therefore have the global estimate
\[
LW \leq - c \gamma W + C \gamma^3 \mathscr{M}_\gamma
\]
for some constant~$C>0$ independent of $\gamma$, which is indeed~\eqref{eq:scaling_Thm2_gamma_geq_1}. 
\end{proof}

\section{Auxiliary estimates and the proof of Proposition~\ref{prop:key}}
\label{sec:auxest}

In this section, we make precise the definition of the operator $A$ introduced in~\eqref{eqn:Adef} and deduce a number of estimates that will be crucial in the proof of Proposition~\ref{prop:key}.  At times, we will also need to make use of a sequence of cutoff functions indexed by $n \in \N$, 
\begin{align}
\label{eqn:cutoffp}
\chi_n\in C^\infty([0, \infty); [0,1]),
\end{align}
satisfying
\begin{align}
\chi_n(t)= \begin{cases}
1 & \text{ if } t\leq n, \\
0 & \text{ if } t\geq n+1,
\end{cases} \qquad \chi_n' \leq 0, \qquad \sup_{n \geq 1} \left\{ \|\chi_n'\|_\infty + \|\chi_n''\|_\infty \right\} < +\infty, 
\end{align}
where $\|\varphi\|_\infty := \sup_{x \in \X}|\varphi(x)|$ for $\varphi$ bounded measurable. 
Depending on the context, we use cutoff functions of the form $\chi_n(U(q))$ or $\chi_n(H(q,p))$.  

In what follows, we will also utilize the gradient dynamics on $\mathscr{O}$ given by~\eqref{eqn:gradient}.  Under Assumptions~\ref{assump:basic} and~\ref{assump:gc}, it follows that the solution $q_t$ of~\eqref{eqn:gradient} is non-explosive and has a unique invariant probability measure given by the $q$-marginal of $\mu$, namely $\mu_{\mathrm{OD}}$. Indeed, invariance of $\mu_{\mathrm{OD}}$ follows by Assumption~\ref{assump:basic}(i) and a direct calculation showing that
\begin{align*}
L^\dagger_{\mathrm{OD}} \left(\rme^{-U}\right) =0,
\end{align*}
where $L^\dagger_{\mathrm{OD}}$ denotes the formal $L^2(dq)$ adjoint of $L_{\mathrm{OD}}$.  Furthermore, uniqueness of $\mu_{\mathrm{OD}}$ follows by Assumption~\ref{assump:basic}(i)-(ii) since $q_t$ is a uniformly non-degenerate diffusion~\cite{kliemann1987recurrence, rey-bellet2006ergodic}.

\subsection{Regularity estimates and moments of $|\nabla U|$}
\label{sec:estimates_tech}

In this section, we establish some basic facts about solutions $\psi$ of the equation 
\begin{align}
\label{eqn:ODpsiphi}
(1-L_{\mathrm{OD}})\psi  = \phi,
\end{align}  
for a given $\phi \in C_\mathrm{b}^\infty(\mathscr{O})$.  We also show that the gradient of $|\nabla U|$ has finite moments with respect to $\mu_{\mathrm{OD}}$.   
 
\begin{Proposition}
\label{prop:solutions}
Suppose that $U$ satisfies Assumptions~\ref{assump:basic} and~\ref{assump:gc}, and let $\phi \in C_\mathrm{b}^\infty(\mathscr{O})$.  Then there exists a unique classical solution $\psi\in C_\mathrm{b}^\infty(\mathscr{O})$ of the equation~\eqref{eqn:ODpsiphi}, which can moreover be expressed by the Green's formula
\begin{align}
\label{formula:psi}
\psi(q)= \E_q\left[ \int_0^\infty \rme^{-s} \phi(q_{s}) \, ds\right],   
\end{align}
where $q_t$ satisfies~\eqref{eqn:gradient}.  
\end{Proposition}

\begin{proof}
We first prove the existence of solutions belonging to the class $C_\mathrm{b}^\infty(\mathscr{O})$. In fact, we show that $\psi$ defined by the formula~\eqref{formula:psi} belongs to this class of functions and satisfies the desired equation.  

Consider the bounded open domains defined in Assumption~\ref{assump:basic}(ii) with boundaries $\partial \mathscr{O}_k= \{q\in \mathscr{O} \, : \, U(q)=k\}$, as well as the sequence of boundary-value problems
\begin{equation}
  \label{eq:eq_psi_k}
  \left\{ \begin{aligned}
    (1-L_{\mathrm{OD}}) \psi_k = \phi & \text{ on } \mathscr{O}_k, \\
    \psi_k=0 & \text{ on } \partial \mathscr{O}_k.  
  \end{aligned} \right.
\end{equation}   
By Assumption~\ref{assump:basic}(iv), there exists $k_* \in \N$ such that $|\nabla U| >0$ on the boundary $\partial \mathscr{O}_k$ for all $k\geq k_*$ and recall that $U\in C^\infty(\mathscr{O})$. The boundary $\partial \mathscr{O}_k$ is therefore $C^\infty$ for $k\geq k_*$. Since the operator $1-L_{\mathrm{OD}}$ is uniformly elliptic and the domain~$\mathscr{O}_k$ is bounded, there exists a unique solution $\psi_k \in C^\infty(\mathscr{O}_k)\cap C_\mathrm{b}(\overline{\mathscr{O}_k})$ to the boundary value problem~\eqref{eq:eq_psi_k} (see for instance~\cite[Chapter~6]{Evans10}). Using It\^{o}'s formula applied to $\rme^{-t} \psi_k(q_t)$ and a nearly identical line of reasoning as in the proof of~\cite[Theorem 9.1.1]{Oks_03}, it follows that this unique solution $\psi_k$ is given by the Feynman--Kac formula
\[
\psi_k(q)=  \E_q\left[ \int_0^{\sigma_k} \rme^{-s}\, \phi(q_{s}) \, ds\right],
\]
where $\sigma_k= \inf\{ t\geq 0\,: \, q_t\notin \mathscr{O}_k\}$.

We next claim that if $\psi$ is defined by the formula~\eqref{formula:psi}, then $\psi_k\rightarrow \psi$ pointwise in~$\mathscr{O}$ as $k\rightarrow \infty$. Let $q\in \mathscr{O}$ and consider $k_* \in \mathbb{N}$ such that $q\in \mathscr{O}_k$ for all $k\geq k_*$ (such an integer~$k_*$ exists by Assumption~\ref{assump:basic}).  It then follows that, for all $k\geq k_*$,
\begin{align}
\label{eqn:pointwiseconvb}
| \psi(q)- \psi_k(q) | = \left|\E_q \left[\int_{\sigma_k}^\infty \rme^{-s} \phi(q_s) \, ds\right] \right| \leq \|\phi\|_\infty \, \E_q \left( \rme^{-\sigma_k} \right).  
\end{align} 
Since $q_t$ is a non-explosive process, $\sigma_k\uparrow \infty$, $\PP_q$-almost surely as $k\rightarrow \infty$.  Using the dominated convergence theorem on the right hand side of~\eqref{eqn:pointwiseconvb}, it follows that $\psi_k(q)\rightarrow \psi(q)$ as $k\rightarrow \infty$, thus establishing the claimed pointwise convergence.

The next step is to show that $\psi$ solves~\eqref{eqn:ODpsiphi} in the sense of distributions on~$\mathscr{O}$. To see this, let $k_*\in \N$ be arbitrary and suppose that $\varphi \in C_\mathrm{c}^\infty(\mathscr{O}_{k_*})$.  Note that $\psi_k$ and $\psi$ are both globally bounded by $\| \phi \|_\infty$.  Letting $(1-L_\mathrm{OD})^{\dagger} = 1-\Delta_q - \mathrm{div}_q\left(\nabla U \cdot\right)$ denote the formal adjoint of $(1-L_\mathrm{OD})$ with respect to the Lebesgue measure on~$\mathscr{O}$ and recalling that $\psi_k \rightarrow \psi$ pointwise on~$\mathscr{O}$, we thus obtain by the dominated convergence theorem again:
\begin{align*}
  \int_\mathscr{O} \psi (1-L_\mathrm{OD})^{\dagger} \varphi & = \lim_{k\rightarrow \infty }\int_{\mathscr{O}_{k_*}} \psi_k (1-L_\mathrm{OD})^{\dagger}\varphi \\
  & = \lim_{k\rightarrow \infty }\int_{\mathscr{O}_{k_*}} \left[(1-L_\mathrm{OD})\psi_k\right] \varphi = \int_{\mathscr{O}_{k_*}} \phi\varphi.
\end{align*}
This shows indeed that $\psi$ solves~\eqref{eqn:ODpsiphi} in the sense of distributions. By results of elliptic regularity, this functions also solves~\eqref{eqn:ODpsiphi} in the classical sense on~$\mathscr{O}$.

Finally, to establish the uniqueness of solutions of~\eqref{eqn:ODpsiphi} belonging to $C^\infty_\mathrm{b}(\mathscr{O})$, let $\psi_1, \psi_2 \in C_\mathrm{b}^\infty(\mathscr{O})$ satisfy~\eqref{eqn:ODpsiphi} and fix~$q\in \mathscr{O}$. Applying It\^{o}'s formula to $\rme^{-t} (\psi_1(q_t) -\psi_2(q_t))$ and evaluating this quantity at time~$t\wedge \sigma_k$, we find that, for all $ t\geq 0, \, k\in \N,$ and $i=1,2$,
\begin{align*}
  \qquad \E_q \left[\rme^{-t\wedge \sigma_k} \psi_i(q_{t\wedge \sigma_k})\right] &= \psi_i(q) + \E_q\left[ \int_0^{t\wedge \sigma_k} (L_{\mathrm{OD}} -1) \psi_i(q_s) \, ds\right] \\
  &= \psi_i(q) - \E_q \left[ \int_0^{t\wedge \sigma_k}  \phi(q_s)\, ds\right].
\end{align*}
Therefore, 
\begin{align*}
  \forall t\geq 0, \quad \forall k \in \N,\qquad \psi_1(q)-\psi_2(q) =  \E_q \left[ \rme^{-t\wedge \sigma_k}\left( \psi_1(q_{t\wedge \sigma_k})- \psi_2(q_{t\wedge \sigma_k}) \right)\right] .
\end{align*}
Hence, we can bound the difference above as 
\begin{align*}
|\psi_1(q)-\psi_2(q)| \leq \left( \|\psi_1 \|_\infty + \| \psi_2 \|_\infty \right) \E_q \left[ \rme^{-t\wedge \sigma_k} \right].  
\end{align*}  
Taking $t\rightarrow \infty$ and then $k\rightarrow \infty$, and using the nonexplosivity of the dynamics, we find that $\psi_1(q)=\psi_2(q)$.  Since $q\in \mathscr{O}$ was arbitrary, we conclude that $\psi_1=\psi_2$ on all of~$\mathscr{O}$.
\end{proof}

\begin{Remark}
\label{rem:Adef}
Note that the proof of the previous result gives a natural way to define the action of the operator 
\begin{align}
A=(1-L_{\rm OD})^{-1} \Pi L_{\rm H}
\end{align}
on the core of functions $\phi \in C^\infty(\X)$ such that $\Pi L_{\rm H} \phi \in C^\infty_\mathrm{b}(\mathscr{O})$, via the stochastic representation 
\begin{align}
  \label{eqn:FK}
  \left(A \phi\right)(q) := \E_q\left[ \int_0^\infty \rme^{-s} (\Pi L_{\rm H}\phi)(q_{s}) \, ds\right],  
\end{align}
where $q_t$ satisfies~\eqref{eqn:gradient}. In particular, $A\phi \in C_\mathrm{b}^\infty(\mathscr{O})$. Observe that the vector space of such functions~$\phi$ is dense in $L^2(d\mu)$ since it contains $C_\mathrm{c}^\infty(\X)$ as well as the constant functions, hence also $C_{\mathrm{c},0}^\infty(\X)$. The operator $A$ above can then be extended to an operator on $L^2(d\mu)$ once we show it is bounded on its dense domain; see Lemma~\ref{lem:standard} below for the latter result.
\end{Remark}

We next establish some moments estimates for~$|\nabla U|$ with respect to the measure~$\mu_{\mathrm{OD}}$, which will prove crucial to control derivatives of~$\chi_n$. Such estimates also allow us to obtain some control on the derivatives of the solution~$\psi\in C_\mathrm{b}^\infty(\mathscr{O})$ to~\eqref{eqn:ODpsiphi}. Part~(ii) of the lemma below is used crucially in the final line of the proof of Proposition~\ref{P:NablaSquared}.

\begin{Lemma}
  \label{lem:diff}
  Suppose that $U$ satisfies Assumptions~\ref{assump:basic} and~\ref{assump:gc}. Then, 
  \begin{itemize}
  \item[(i)] $|\nabla U|\in L^r(d\mu_{\mathrm{OD}})$ for any $r\geq 1$;
  \item[(ii)] for any $\phi \in C_\mathrm{b}^\infty(\mathscr{O})$ and $r\geq 1$, the unique solution $\psi \in C^\infty_\mathrm{b}(\mathscr{O})$ to~\eqref{eqn:ODpsiphi} belongs to $H^1(d\mu_{\mathrm{OD}})\cap H^1(|\nabla U|^r \, d\mu_{\mathrm{OD}})$.
  \item[(iii)]  for any $\delta >0$, $\rme^{-\delta U} \in L^1(dq)$.  
  \end{itemize}
\end{Lemma}
\begin{proof}

  To prove~(i), fix $r\geq 3$ and consider the sequence of functions 
  \[
  \varphi_n(q,p) = p \cdot \nabla U |\nabla U|^{r-2} \chi_n(U),  
  \]
 where $\chi_n$ is defined in~\eqref{eqn:cutoffp}. A simple computation shows that 
 \begin{align*}
   L \varphi_n(q,p) = \sum_{i,j=1}^d p_i p_j \partial_{q_j}(\partial_{q_i}U |\nabla U|^{r-2} \chi_n(U)) - |\nabla U|^{r} \chi_n(U) - \gamma\varphi_n(q,p).  
 \end{align*}
 Note that the function $L \varphi_n$ belongs to~$L^1(\mu)$ since~$\chi_n(U)$ has compact support. Now, since the integral of $p_ip_j$ against~$\mu_{\mathrm{OD}}$ is~$\delta_{ij}$ because $\mu_\mathrm{OD}$ is a standard centered Gaussian, while~$\varphi_n$ has average~0 with respect to~$\mu$, 
 \begin{align*}
 0 = \int_{\X} L \varphi_n \, d\mu = \int_\mathscr{O} \sum_{j=1}^d  \partial_{q_j} \left[ \partial_{q_j} U |  \nabla U|^{r-2} \chi_n(U)\right]  - |\nabla U|^{r} \chi_n(U) \, d\mu_{\mathrm{OD}}.  
 \end{align*}
 Since $\chi_n' \leq 0$ and in view of Assumption~\ref{assump:gc}, there exists $C >0$ such that  
 \begin{align*}
   \sum_{j=1}^d \partial_{q_j} \left[ \partial_{q_j} U |  \nabla U|^{r-2} \chi_n(U)\right] \leq \left(\frac{1}{2} |\nabla U|^{r} + C |\nabla U|^{r-2} \right) \chi_n(U).  
 \end{align*}
 Combining the previous two estimates we find that 
 \begin{align*}
   \forall n \geq 1, \qquad \frac12 \int_\mathscr{O} |\nabla U|^{r} \chi_n(U)  \, d\mu_{\mathrm{OD}} \leq C \int_\mathscr{O} |\nabla U|^{r-2} \chi_n(U)  \, d\mu_{\mathrm{OD}}.
  \end{align*}
 In view of the following lower bound, obtained by a H\"older inequality,
 \[
 \int_\mathscr{O} |\nabla U|^{r} \chi_n(U)  \, d\mu_{\mathrm{OD}} \geq \left(\int_\mathscr{O} |\nabla U|^{r-2} \chi_n(U)  \, d\mu_{\mathrm{OD}}\right)^{r/(r-2)}\left(\int_\mathscr{O} \chi_n(U)  \, d\mu_{\mathrm{OD}}\right)^{-2/(r-2)},
 \]
 we find that
 \[
 \left(\int_\mathscr{O} |\nabla U|^{r-2} \chi_n(U)  \, d\mu_{\mathrm{OD}}\right)^{2/(r-2)} \leq 2C \left(\int_\mathscr{O} \chi_n(U)  \, d\mu_{\mathrm{OD}}\right)^{2/(r-2)} \leq 2C,
 \]
 from which the result (i) follows. 

 For part~(ii), let us fix $\phi \in C_\mathrm{b}^\infty(\mathscr{O})$ and consider the unique solution $\psi \in C^\infty_\mathrm{b}(\mathscr{O})$ to~\eqref{eqn:ODpsiphi}. In the remainder of this proof, because all functions are functions of $q$ only, we use $\| \cdot \|$ to denote the canonical norm on $L^2(d\mu_{\mathrm{OD}})$. We first show that $\psi \in H^1(d\mu_{\mathrm{OD}})$. A Cauchy--Schwarz inequality and integration by parts lead to
 \begin{align*}
\| \psi \| \, \| \phi \| &\geq \int_\mathscr{O} \phi \psi \chi_n(U) \, d\mu_{\rm OD} = \int_\mathscr{O} (1- L_\mathrm{OD}) \psi  \cdot \psi  \chi_n(U)\, d\mu_{\mathrm{OD}}\\
& = \int_\mathscr{O} \psi^2 \chi_n(U) \,  d\mu_{\mathrm{OD}}  + \int_\mathscr{O} \nabla_q \psi \cdot \nabla_q \left( \psi \chi_n(U) \right) d\mu_{\mathrm{OD}}\\
& = \int_\mathscr{O} \psi^2 \chi_n(U) \,  d\mu_{\mathrm{OD}}  + \int_\mathscr{O} |\nabla_q \psi|^2 \chi_n(U) \, d\mu_{\mathrm{OD}} - \frac12 \int_\mathscr{O} \psi^2 L_{\mathrm{OD}}\left( \chi_n(U) \right) d\mu_{\mathrm{OD}}.
\end{align*}
The last integral on the right hand side converges to~0 as $n\to+\infty$ by a dominated convergence argument, since $\psi$ is bounded, while $L_{\mathrm{OD}} (\chi_n(U)) = \chi_n''(U)|\nabla U|^2 + \chi_n'(U)(\Delta U-|\nabla U|^2)$ converges pointwise to~0 and $|\nabla U|^2,\Delta U \in L^1(d\mu_{\rm OD})$ by Assumption~\ref{assump:gc} and part~(i). We therefore obtain the bound $\| \psi \|^2 + \|\nabla_q\psi\|^2 \leq \| \psi \| \|\phi\|$ as $n\to+\infty$, which proves that $\psi \in H^1(d\mu_{\mathrm{OD}})$.  

To see that $\psi \in H^1(|\nabla U|^r \, d\mu_{\mathrm{OD}})$ for any $r\geq 1$, we may assume without loss of generality that $r\geq 3$. By manipulations similar to the ones used to prove the $H^1(\mu_\mathrm{OD})$ bound, we obtain
\begin{align*}
& \left\| \psi |\nabla U|^r \right\|\| \phi \| \geq \int_\mathscr{O} (1- L_\mathrm{OD}) \psi  \cdot \psi |\nabla U|^r \chi_n(U)\, d\mu_{\mathrm{OD}}\\
& \ \ = \int_\mathscr{O} \psi^2 \chi_n(U) \, |\nabla U|^r d\mu_{\mathrm{OD}}  + \int_\mathscr{O} |\nabla_q\psi |^2 |\nabla U|^r \chi_n(U) \, d\mu_{\mathrm{OD}} + \int_\mathscr{O} \psi \nabla_q \psi \cdot \nabla\left[|\nabla U|^r \chi_n(U)\right] d\mu_{\mathrm{OD}}. 
\end{align*}
The last integral can be bounded as
\[
\begin{aligned}
  & \left| \int_\mathscr{O} \psi \nabla_q \psi \cdot \nabla\left[|\nabla U|^r \chi_n(U)\right] d\mu_{\mathrm{OD}} \right| \\
  & \qquad \leq \left|\int_\mathscr{O} \psi \nabla_q \psi \cdot \nabla\left[|\nabla U|^r\right] \chi_n(U)\,d\mu_{\mathrm{OD}} \right| + \left|\int_\mathscr{O} \psi \nabla_q \psi \cdot \nabla U |\nabla U|^r \chi'_n(U)\,d\mu_{\mathrm{OD}}\right| \\
  & \qquad \leq \|\nabla_q \psi\|\left( \left\| \psi \nabla\left[|\nabla U|^r\right] \right\| +  \left\|\psi \nabla U |\nabla U|^r \chi'_n(U)\right\|\right).
\end{aligned}
\]
By part~(i) and Assumption~\ref{assump:gc}, the last term in the above inequality converges to~0 as $n\to+\infty$, while the other term is bounded. We therefore obtain, in the limit $n \to +\infty$,
\begin{align*}
  \| \psi \|^2_{L^2(|\nabla U|^r \, d\mu_{\mathrm{OD}})}+ \| \nabla_q \psi \|^2_{L^2(|\nabla U|^r \, d\mu_{\mathrm{OD}})} & \leq \| \psi |\nabla U|^r \|\| \phi \| + \|\nabla_q \psi\| \left\| \psi \nabla\left(|\nabla U|^r\right) \right\| \\
  & \leq \|\psi\|_\infty \left( \| |\nabla U|^r \| \| \phi \| + \|\nabla_q \psi\| \left\| \nabla\left(|\nabla U|^r\right) \right\| \right) < +\infty,
\end{align*}
which concludes the proof of part (ii).  

In order to deduce part (iii) of the result, it is enough to show that $\mathrm{e}^{\eta U} \in L^1(d\mu_{\mathrm{OD}})$ for any $\eta \in (0,1)$.  We do this by applying~\cite[Proposition~5.1]{HairMat_09} and showing that, for any $\eta \in (0,1)$, it holds $L_{\mathrm{OD}} (\mathrm{e}^{\eta U}) \rightarrow -\infty$ as $U\rightarrow \infty$. Fixing $\eta \in (0,1)$, a direct calculation shows that
\begin{align*}
\frac{L_\mathrm{OD} \left(\mathrm{e}^{\eta U}\right)}{\eta\mathrm{e}^{\eta U}} = -(1-\eta) |\nabla U|^2 + \Delta U,  
\end{align*}
the function on the right hand side going to~$-\infty$ as $U(q) \to +\infty$ by Assumptions~\ref{assump:basic}(iv) and~\ref{assump:gc}. 
\end{proof}

\subsection{$L^2$-bounds and the elliptic regularity estimate}
\label{sec:elliptic_reg}

We deduce here various $L^2$ estimates involving the operator $A$, as well as an elliptic regularity estimate similar to the one in~\cite{dolbeault2015hypocoercivity}, from the estimates proved in Section~\ref{sec:estimates_tech}.  

\begin{Lemma}
  \label{lem:standard}
  Suppose that $U$ satisfies Assumptions~\ref{assump:basic} and~\ref{assump:gc}. Then, for any $\phi \in C_\mathrm{c}^\infty(\X)$,
  \begin{align}
  \| A \phi \| & \leq \frac{1}{2}\| (1-\Pi)\phi \|, \label{eq:ineq_A}\\
  |\langle L^*A \phi, \phi\rangle| & \leq\|(1-\Pi)\phi\|^2,\label{eq:ineq_LA}\\
  |\langle A L_{\mathrm{OU}} \phi, \phi \rangle| & \leq \frac{1}{2} \| (1-\Pi) \phi\| \| \Pi\phi \|. \label{eq:ineq_ALOU}
  \end{align}
\end{Lemma}

The extension of the same inequalities to functions in $C_{\mathrm{c},0}^\infty(\X)$ (Corollary~\ref{cor:C_{c,0}} below) is immediate since functions in this space are a constant shift of functions in $C_\mathrm{c}^\infty(\X)$, and $A c=L_\mathrm{OU}c=Lc=0$ for any constant function~$c$. 
\begin{Corollary}
\label{cor:C_{c,0}}
Suppose that $U$ satisfies Assumptions~\ref{assump:basic} and~\ref{assump:gc}.  Then for any $\phi \in C_{\mathrm{c},0}^\infty(\X)$ the estimates~\eqref{eq:ineq_A}-\eqref{eq:ineq_LA}-\eqref{eq:ineq_ALOU} are also satisfied. 
\end{Corollary}

\begin{Remark}
  Note that, by the definition of~$A$ given in Remark~\ref{rem:Adef}, part~(i) of Proposition~\ref{prop:key} readily follows from the above result by extending $A$ to be a bounded operator on $L^2(d\mu)$.    
\end{Remark}

\begin{proof}[Proof of Lemma~\ref{lem:standard}]
  From an algebraic viewpoint, the proof follows the proof of~\cite[Lemma~1]{dolbeault2015hypocoercivity}. Let $\phi \in C_\mathrm{c}^\infty(\X)$.  We start by proving~\eqref{eq:ineq_A}. By Proposition~\ref{prop:solutions}, we know that $\psi = A\phi \in C_\mathrm{b}^\infty(\mathscr{O})$. For simplicity of notation in the arguments that follow, we introduce $\langle f, g\rangle_n:= \langle f, g \chi_n (H) \rangle$ and $\| f\|^2_n = \langle f,f\rangle_n$.  Noting that $\psi + (L_\mathrm{H} \Pi)^*(L_\mathrm{H} \Pi) \psi =-(L_\mathrm{H} \Pi)^* \phi$, we obtain 
  \begin{align}
    \label{eq:psi_LH_psi}
\| \psi\|_{n}^2 + \langle (L_\mathrm{H}\Pi)^* L_\mathrm{H} \Pi \psi, \psi\rangle_n  =-\langle (L_\mathrm{H}\Pi)^* \phi ,\psi \rangle_n.  
 \end{align}       
 First observe that (using that~$\psi$ is a function of~$q$ only, so that $\Pi(\psi \chi_n(H)) = \psi \Pi \chi_n(H)$)
 \begin{align*}
   \langle (L_\mathrm{H}\Pi)^* L_\mathrm{H}\Pi \psi, \psi\rangle_n = \|L_\mathrm{H} \Pi \psi\|_{L^2(\Pi \chi_n(H) d\mu)}^2 + \langle L_\mathrm{H} \Pi \psi, \psi L_\mathrm{H} \Pi \chi_n(H)\rangle =: T_1(n) + T_2(n).  
 \end{align*}
 Recalling $\Pi \psi = \psi$, and using $2\psi L_\mathrm{H} \psi = L_\mathrm{H}(\psi^2)$ as well as an integration by parts, we obtain 
 \begin{align*}
   2 T_2(n) = - \left\langle\psi^2, L_\mathrm{H}^2 \Pi \chi_n(H) \right\rangle. 
 \end{align*}
 Now,
 \[
 \begin{aligned}
   L_\mathrm{H}^2 \Pi \chi_n(H) & = L_\mathrm{H}\left( p^T \Pi \left[ \chi_n'(H) \nabla U\right] \right) \\
   & = -|\nabla U|^2 \Pi \left[ \chi_n'(H)\right] + p^T \left( \Pi\left[\chi_n'(H)\right] \nabla^2 U + \Pi\left[ \chi_n''(H) \right] \nabla U \otimes \nabla U\right) p.
 \end{aligned}
 \]
 Applying Lemma~\ref{lem:diff} and recalling that~$\psi \in C^\infty_\mathrm{b}(\mathscr{O})$, we see that $T_2(n)\rightarrow 0$ as $n\rightarrow \infty$.  On the other hand, 
   \begin{align*}
   -\langle (L_\mathrm{H}\Pi)^* \phi ,\psi \rangle_n= - \langle (1-\Pi) \phi, (L_\mathrm{H} \Pi \psi) \Pi \chi_n(H) \rangle - \langle \phi, \psi  L_\mathrm{H} \Pi \chi_n(H) \rangle=: -T_3(n)- T_4(n),  
   \end{align*}
   where we used that $\Pi L_\mathrm{H} \Pi = 0$ in the first equality. By the same reasoning used above to show that~$T_2(n)\rightarrow 0$ as $n\rightarrow \infty$ via Lemma~\ref{lem:diff}, it also holds that $T_4(n)\rightarrow 0$ as $n\rightarrow \infty$. Putting these estimates together and using Young's inequality for products, we find that
   \begin{align*}
     &  \|\psi\|^2_n  = -\|L_\mathrm{H} \Pi \psi \|_{L^2(\Pi \chi_n(H) d\mu)}^2 - \langle (1-\Pi) \phi, (L_\mathrm{H} \Pi \psi ) \Pi \chi_n(H) \rangle -T_2(n)-T_4(n) \\   
     & \leq  -\|L_\mathrm{H} \Pi \psi \|_{L^2(\Pi \chi_n(H) d\mu)}^2   + \frac{1}{4} \|(1-\Pi) \phi \|^2_{L^2(\Pi \chi_n(H) d\mu)}  + \| L_\mathrm{H} \Pi \psi \|_{L^2(\Pi \chi_n(H) d\mu)}^2 -T_2(n)-T_4(n) \\
     &\leq \frac{1}{4} \|(1-\Pi) \phi \|^2 + |T_2(n)| + |T_4(n)|, 
   \end{align*}
  where on the last line we used the fact that $\chi_n(H) \in [0,1]$. The claimed estimate~\eqref{eq:ineq_A} follows by taking the limit~$n\rightarrow \infty$.  

  We next claim that
  \begin{align}
    \label{eqn:bound2}
    \forall \phi \in C_{\mathrm{c}}^\infty(\X), \qquad \|L_\mathrm{H} \Pi A \phi \|= \| L_\mathrm{H} A \phi\|\leq \|(1-\Pi) \phi\|.
  \end{align}
  To prove this inequality, we use Young's inequality to estimate $T_3$ as
  \begin{align*}
    |T_3(n)| &\leq \frac{1}{2}\| (1-\Pi)\phi\|^2_{L^2(\Pi \chi_n(H) \, d\mu)} + \frac{1}{2}\| L_\mathrm{H} \Pi \psi  \|^2_{L^2(\Pi \chi_n(H) \, d\mu)}\\
    & \leq  \frac{1}{2}\| (1-\Pi)\phi\|^2+ \frac{1}{2}\| L_\mathrm{H} \Pi \psi  \|^2_{L^2(\Pi \chi_n(H) \, d\mu)},
  \end{align*}
  where on the last line we used the fact that $\chi_n(H) \in [0,1]$. Then, 
  \begin{align*}
    \| L_\mathrm{H} \Pi \psi \|^2_{L^2(\Pi \chi_n(H) d\mu)} + \|\psi\|_n^2 &= - T_2(n)-T_3(n)-T_4(n)\\
    &\leq    \frac{1}{2}\| (1-\Pi)\phi\|^2+ \frac{1}{2}\| L_\mathrm{H} \Pi \psi  \|^2_{L^2(\Pi \chi_n(H) \, d\mu)} +|T_2(n)|+|T_4(n)|.
  \end{align*}  
  The bound~\eqref{eqn:bound2} follows after taking $n\rightarrow \infty$. Since $\Pi A = A$, this finally gives
\begin{equation}  
\label{eqn:bound3}
  \begin{aligned}
        \forall \phi \in C_\mathrm{c}^\infty(\X), \qquad | \langle L_\mathrm{H} A \phi, \phi  \rangle| & =| \langle L_\mathrm{H} A \phi, (1-\Pi)\phi  \rangle|= | \langle L^* A \phi, (1-\Pi)\phi \rangle | \\
    & \leq \|(1-\Pi)\phi\|^2.  
  \end{aligned}
  \end{equation}
      
  Finally, a short calculation shows that $AL_{\mathrm{OU}}=-A$, so that, for any $\phi \in C_\mathrm{c}^\infty(\X)$,
  \[
    |\langle A L_{\mathrm{OU}} \phi, \phi \rangle| = |\langle A L_{\mathrm{OU}} \phi, \Pi\phi \rangle| = |\langle A \phi, \Pi \phi \rangle| \leq \frac{1}{2} \| (1-\Pi) \phi \| \| \Pi \phi \|.  
  \] 
  This concludes the proof of the lemma.  
 \end{proof}

We next turn to the elliptic regularity estimate. Note that the proof follows arguments similar to those in~\cite{CLW_19}, but with a slightly different application of the more general bound in Assumption~\ref{assump:gc}.  We also refer the reader to~\cite[Lemma~4]{BFS_20} for a related result.   

\begin{Proposition}
  \label{P:NablaSquared}
  Suppose that the potential $U$ satisfies Assumptions~\ref{assump:basic} and~\ref{assump:gc}.  Let $c_1\in (0,1)$,  $C_2>0$ be such that the bound in~\eqref{eqn:c1C2} is satisfied and let $\varepsilon >0$ satisfy~\eqref{eqn:eps}.  Let $C_\varepsilon$ be the corresponding constant for this choice of $\varepsilon$ in Assumption~\ref{assump:gc}, and set $\xi_\varepsilon = \eta_\varepsilon^2/2$ with $\eta_\varepsilon$ defined in~\eqref{eqn:constants}. Then, for any $\phi \in C_\mathrm{b}^\infty(\mathscr{O})$,
  \[
  \left\|\nabla_q^2\psi\right\|_{L^2(d\mu_{\mathrm{OD}})}^2\leq \xi_\varepsilon \|\phi\|_{L^2(d\mu_{\mathrm{OD}})}^2,
  \]
  where $\psi\in C_\mathrm{b}^\infty(\mathscr{O})$ is the unique classical bounded solution of $(1-L_\mathrm{OD}) \psi = \phi$ as given by Proposition~\ref{prop:solutions}.   
\end{Proposition}

\begin{proof}[Proof of Proposition~\ref{P:NablaSquared}]
  Since all the functions which appear in this proof are functions of~$q$ only, and since the norm $\| \cdot \|$ coincides with the norm $\| \cdot \|_{L^2(d\mu_{\mathrm{OD}})}$ in this case, we use in this proof the notation~$\| \cdot \|$ for $\| \cdot \|_{L^2(d\mu_{\mathrm{OD}})}$. We first prove that
  \begin{equation}
    \label{eq:tech_estimate_prop5}
    \forall \varphi \in C_\mathrm{c}^\infty(\mathscr{O}),
    \qquad
    \left\| \nabla^2_q \varphi \right\|^2 \leq \xi_\varepsilon \| (1- L_\mathrm{OD}) \varphi \|^2.
  \end{equation}
  We then conclude the desired estimate using a density argument afforded by Lemma~\ref{lem:diff}.

  Fix $\varphi \in C_\mathrm{c}^\infty(\mathscr{O})$. Bochner's formula gives 
  \[
  \left\|\nabla_q^2\varphi\right\|^2 = \sum_{i,j=1}^d \left\|\partial_{q_i,q_j}^2 \varphi \right\|^2 = \left\|\nabla_q^*\cdot\nabla_q\varphi\right\|^2-\left\langle\nabla_q\varphi,\nabla^2U(q)\nabla_q\varphi\right\rangle.
  \]
Using Assumption~\ref{assump:gc} with the choice of $\varepsilon, C_\varepsilon$ as in the statement of the result leads to
\begin{equation}
  \label{eq:Bochner_bound}
  \left\|\nabla_q^2\varphi\right\|^2\leq \|\nabla_q^*\cdot\nabla_q\varphi\|^2+\varepsilon\||\nabla U(q)||\nabla_q\varphi|\|^2+C_\varepsilon\|\nabla_q\varphi\|^2.
\end{equation}
To control $\||\nabla U(q)||\nabla_q\varphi|\|$, we use the the structure of the Gibbs measure~$\mu$ to write, for any $f \in C_\mathrm{c}^\infty(\mathscr{O})$,
\[
\|f\nabla U(q)\|^2 = \left\langle \nabla_q(f^2),\nabla U(q)\rangle+\langle f^2,\Delta U(q)\right\rangle = 2 \left\langle \nabla_qf,f\nabla U(q)\rangle+\langle f^2,\Delta U(q) \right\rangle.
\]
With the previous inequality, the estimate~\eqref{eqn:c1C2} and Young's inequality, we obtain, for any~$\eta \in (0,1-c_1)$,
\[
\|f\nabla U(q)\|^2\leq (c_1+\eta)\|f\nabla U(q)\|^2+\eta^{-1}\|\nabla_qf\|^2+C_2\|f\|^2,
\]
so that, by rearranging the previous expression, 
\begin{align}
\label{eqn:phibound}
\|f\nabla U(q)\|^2\leq \frac{\eta^{-1}}{1-c_1-\eta}\|\nabla_qf\|^2+\frac{C_2}{1-c_1-\eta}\|f\|^2.
\end{align}
The above estimate can be extended by density to any $f \in H^1(\mu_{\mathrm{OD}})$. Now, if $\varphi \in C^\infty_\mathrm{c}(\mathscr{O})$, then for all $\zeta>0$ we define $g_\zeta:=\sqrt{|\nabla_q\varphi|^2+\zeta}\in H^1(\mu_{\mathrm{OD}})$. Therefore, by applying~\eqref{eqn:phibound} to~$g_\zeta$, 
\[
\| |\nabla U | |\nabla_q \varphi | \|^2\leq \left\| | \nabla U(q)|g_\zeta\right\|^2 \leq \frac{\eta^{-1}}{1-c_1-\eta}\left\|\nabla_q^2\varphi\right\|^2+\frac{C_2}{1-c_1-\eta}\|g_\zeta\|^2.
\]
Taking $\zeta\rightarrow 0$, we arrive at the following inequality 
\[
\| |\nabla U|  | \nabla_q \varphi | \|^2 \leq \frac{\eta^{-1}}{1-c_1-\eta}\left\|\nabla_q^2\varphi\right\|^2+\frac{C_2}{1-c_1-\eta}\|\nabla_q \varphi\|^2.   
\]
Combining this estimate with~\eqref{eq:Bochner_bound} gives
\[
\left\|\nabla_q^2\varphi\right\|^2\leq \|\nabla_q^*\cdot\nabla_q\varphi\|^2+\left(C_{\varepsilon}+\frac{\varepsilon C_2}{1-c_1-\eta}\right)\|\nabla_q\varphi\|^2+\frac{\varepsilon\eta^{-1}}{1-c_1-\eta}\left\|\nabla_q^2\varphi\right\|^2,
\]
and, after rearranging the various terms and provided~$\eta(1-c_1-\eta)>\varepsilon$, 
\[
\left\|\nabla_q^2\varphi\right\|^2\leq \left(1-\frac{\varepsilon\eta^{-1}}{1-c_1-\eta}\right)^{-1}\left[\|\nabla_q^*\cdot\nabla_q\varphi\|^2+\left(C_{\varepsilon}+\frac{\varepsilon C_2}{1-c_1-\eta}\right)\|\nabla_q\varphi\|^2\right].
\]
Choosing~$\eta=(1-c_1)/2$, and noting that
\[
  \left\| (1+\nabla^*_q \cdot \nabla_q) \varphi \right\|^2=\|\varphi\|^2+2\left\|\nabla_q\varphi\right\|^2+\left\|\nabla_q^*\cdot\nabla_q\varphi\right\|^2,
\]
the inequality reduces to
\begin{align}
  \left\|\nabla_q^2\varphi\right\|^2 &\leq \max\left\{1,\frac{1}{2}\left(C_\varepsilon+\frac{2C_2\varepsilon}{1-c_1}\right)\right\}\left(1-\frac{4\varepsilon}{(1-c_1)^2}\right)^{-1}\left\|(1+\nabla_q^*\cdot\nabla_q)\varphi\right\|^2 \nonumber\\
  &= \xi_\varepsilon \left\| (1+ \nabla_q^* \cdot \nabla_q ) \varphi \right\|^2. \label{eqn:baseest} 
\end{align}
Since $-L_\mathrm{OD} = \nabla_q^* \cdot \nabla_q$, this indeed provides the claimed bound~\eqref{eq:tech_estimate_prop5}. 

Now consider~$\phi \in C_\mathrm{b}^\infty(\mathscr{O})$ and $\psi \in C_\mathrm{b}^\infty(\mathscr{O})$ the unique classical bounded solution of $(1-L_\mathrm{OD}) \psi = \phi$.  Plugging $\varphi = \psi  \chi_n(U) \in C^\infty_\mathrm{c}(\mathscr{O})$ into~\eqref{eqn:baseest} gives 
\begin{align*}
\forall n \geq 1, \qquad \left\| \nabla_q^2 (\psi \chi_n(U)) \right\|^2 \leq \xi_\varepsilon \| (1-L_\mathrm{OD}) (\psi \chi_n(U)) \|^2.
\end{align*}
In view of Lemma~\ref{lem:diff}(ii) and Assumption~\ref{assump:gc}, we therefore obtain, by dominated convergence,
\begin{align*}
\lim_{n\rightarrow \infty} \left\| \nabla_q^2 (\psi \chi_n(U)) \right\|^2 = \left\| \nabla_q^2 \psi \right\|^2 \leq \xi_\varepsilon \| \phi \|^2 = \xi_\varepsilon \lim_{n\rightarrow \infty} \left\| (1-L_\mathrm{OD}) (\psi \chi_n(U)) \right\|^2, 
\end{align*} 
which concludes the proof.  
\end{proof}

With the previous estimate at hand, we conclude this subsection with one final estimate.  This is the final ingredient needed in the proof of Proposition~\ref{prop:key} parts~(ii) and~(iii).  

\begin{Lemma}
  \label{lem:weight}
  Let $U$ satisfy Assumptions~\ref{assump:basic} and~\ref{assump:gc}, and let $\eta_\varepsilon$ be as in~\eqref{eqn:constants}. Then, 
  \begin{align}
    \forall \phi \in C_{\mathrm{c},0}^\infty(\X),
    \qquad
    |\langle AL_{\mathrm{H}}(1-\Pi)\phi,\phi\rangle|\leq\eta_\varepsilon \|(1-\Pi)\phi\|\|\Pi\phi\|.
  \end{align}    
\end{Lemma}  

\begin{proof}
  Let $\phi \in C_{\mathrm{c},0}^\infty(\X)$. Recall that~$L_{\mathrm{OD}} = -\nabla_q^*\cdot\nabla_q$. A direct calculation yields
  \[
  \big[(AL_{\mathrm{H}})(1-\Pi)\big]^*=\left((p\cdot\nabla_q)^2-\Delta_q\right)\left(1+\nabla_q^*\cdot\nabla_q\right)^{-1}\Pi.
  \]
  Hence
  \begin{align*}
    |\langle AL_{\mathrm{H}}(1-\Pi)\phi,\phi\rangle|&=\big|\big\langle(1-\Pi)\phi,\big[AL_{\mathrm{H}}(1-\Pi)\big]^*\phi\big\rangle\big|\\
    &= \left|\left\langle(1-\Pi)\phi,\left((p\cdot\nabla_q)^2-\Delta_q\right)\left(1+\nabla_q^*\cdot\nabla_q\right)^{-1}\Pi\phi\right\rangle\right| \\
    &\leq \| (1- \Pi) \phi \| \left\| \left((p\cdot\nabla_q)^2-\Delta_q\right)\left(1+\nabla_q^*\cdot\nabla_q\right)^{-1}\Pi\phi \right\|.
  \end{align*}
  Since $\psi=(1+\nabla_q^*\cdot\nabla_q)^{-1}\Pi\phi\in C_\mathrm{b}^\infty(\mathscr{O})$ is a function of the variable~$q$ only, we have, from the structure of $\mu$,
  \begin{align*}
    &    \left\| \left( (p\cdot\nabla_q)^2-\Delta_q \right)\psi\right\|^2 = \left\|\sum_{i,j=1}^d p_ip_j\partial_{q_i,q_j}^2\psi -\Delta_q \psi \right\|^2\\
    & \qquad =\| \Delta_q \psi\|^2 + \sum_{i,j,k, \ell=1}^d \int_\X p_ip_j p_k p_\ell \left(\partial^2_{q_i q_j}\psi\right)\left(\partial_{q_k q_\ell}^2 \psi\right) d\mu - 2\sum_{i,j=1}^d \left\langle p_i p_j\partial_{q_i,q_j}^2\psi,\Delta_q \psi\right\rangle \\
    & \qquad =\| \Delta_q \psi\|^2 + \left(\sum_{i=j,k=\ell}+\sum_{i=k,j=\ell}+\sum_{i=\ell,j=k}-2\sum_{i=j=k=\ell}\right) \int_\X p_ip_j p_k p_\ell \left(\partial^2_{q_i q_j}\psi\right)\left(\partial_{q_k q_\ell}^2 \psi\right) d\mu\\
    &\qquad\qquad- 2\sum_{i=1}^d \left\langle p_i^2\partial_{q_i}^2\psi,\Delta_q \psi\right\rangle \\
    & \qquad = -\|\Delta_q\psi\|^2+\sum_{i,j=1}^d \int_\X p_i^2 p_j^2\left[2\left(\partial_{q_i q_j}^2\psi\right)^2+\left(\partial_{q_i}^2\psi\right)\left(\partial_{q_j}^2\psi\right)\right] d\mu - 2\sum_{i=1}^d \int_\X p_i^4 \left(\partial_{q_i}^2\psi\right)^2 \, d\mu \\
    & \qquad = 2\sum_{i,j=1}^d \left\|\partial_{q_i q_j}^2\psi\right\|^2 + \sum_{i,j=1}^d \left\langle \partial_{q_i}^2\psi,\partial_{q_j}^2\psi \right\rangle -\|\Delta_q\psi\|^2 = 2 \left\|\nabla_q^2\psi\right\|^2\leq 2 \xi_\varepsilon\| \Pi \phi\|^2, 
  \end{align*}
  where the last inequality follows by Proposition~\ref{P:NablaSquared}.  Also, in the above, we can calculate the terms involving products of the $p_i$'s since $\psi$ is a function of $q$ only, and $\mu$ is a product measure with a standard Gaussian in $p$ as its $p$-marginal.  In particular, along the $i=j$ terms of the final summation we recover the fourth-moment of the Gaussian measure in $p$ (\emph{i.e.} a coefficient of~$3$), while for the $i\neq j$ terms of the final sum we recover the variance, which is~$1$. It thus now follows $|\langle AL_{\mathrm{H}}(1-\Pi)\phi,\phi\rangle|\leq \sqrt{2 \xi_\varepsilon} \|(1-\Pi)\phi\|\|\Pi\phi\|= \eta_\varepsilon \| (1-\Pi) \phi \| \| \Pi \phi \|$, as claimed.
\end{proof}

\subsection{Proof of Proposition~\ref{prop:key}}
\label{sec:proof_roposition_prop:key}

We conclude this section by giving a proof of Proposition~\ref{prop:key}.  This is the last theoretical step in the proof of our main general results, namely Theorem~\ref{thm:premain} part~(i) and Theorem~\ref{thm:main} part~(i).     

\begin{proof}[Proof of Proposition~\ref{prop:key}]

Part~(i) of the result is a direct consequence of Lemma~\ref{lem:standard} and a density argument. Let us next turn to part~(ii) of the result. For any $\phi \in C_\mathrm{b}^\infty(\X)$ and $n\in \N$, we define (recalling the definition~\eqref{eqn:cutoffp} for~$\chi_n(H)$)
\begin{align}
P_t^n\phi= \chi_n(H) P_t \phi, \qquad \sP\phi=\int_\X \phi \, d\mu, \qquad  \sPp\phi = \phi - \sP \phi.  
\end{align}
Consider now $\phi \in C_{\mathrm{c},0}^\infty(\X)$. Since $A\sPp \varphi = A\varphi$ for any $\varphi \in C_\mathrm{c}^\infty(\X)$, and recalling that $\frac{d}{dt} P_t\phi = LP_t \phi = P_t L\phi \in C^\infty([0,+\infty) \times \X)$  by hypoellipticity~\cite{Hor_67},
\[
\frac{d}{dt} \left\langle  A \sPp P_t^n \phi, \sPp P_t^n \phi \right\rangle = T_1 + T_2+ T_3,
\]
with
\[
\begin{aligned}
  T_1 &= \left\langle  A [(L P_t \phi) \chi_n(H)], \sPp P_t^n \phi \right\rangle, \\
  T_2 &= \left\langle  A \sPp P_t^n \phi, (LP_t \phi) \chi_n(H) \right\rangle, \\
  T_3 &= -\sP\left[ A P_t^n \phi\right]\sP \left[(LP_t \phi) \chi_n(H)\right]. 
\end{aligned}
\]
Using the identity $L(fg)=fLg+ gLf+2\gamma \nabla_p f \cdot \nabla_p g$, we can write
  \begin{align*}
    T_1 &= \left\langle A L \sPp P_t^n \phi , \sPp P_t^n \phi\right\rangle - \left\langle A [P_t \phi L \chi_n(H)], \sPp P_t^n \phi \right\rangle \\
    & \qquad - 2\gamma \left\langle A[ \nabla_p(P_t \phi) \cdot \nabla_p \chi_n(H)], \sPp P_t^n \phi \right\rangle \\
&=: T_1' + R_1(n)+ R_2(n).   
  \end{align*}
  In order to help estimate $R_2(n)$, note that, for $f\in C^\infty(\X)$ and $g \in C_\mathrm{c}^\infty(\X)$,
  \[
A \left(\nabla_p f \cdot \nabla_p g\right) = -\sum_{i=1}^d (1-L_\mathrm{OD})^{-1}\partial_{q_i}^* \Pi[\partial_{p_i}\left(\nabla_p f \cdot \nabla_p g\right)].
\]
Next, using nearly identical arguments to those in Lemma~\ref{lem:standard}, observe that the operators $\mathcal{T}_i = (1-L_\mathrm{OD})^{-1}\partial_{q_i}^*$ are bounded on~$L^2(d\mu_\mathrm{OD})$ with norms smaller than~1 since
\[
\begin{aligned}
  \sum_{i=1}^d \mathcal{T}_i \mathcal{T}_i^* & =\sum_{i=1}^d (1-L_\mathrm{OD})^{-1}\partial_{q_i}^* \left[(1-L_\mathrm{OD})^{-1}\partial_{q_i}^*\right]^* = (1-L_\mathrm{OD})^{-1}\left(\sum_{i=1}^d \partial_{q_i}^* \partial_{q_i}\right) (1-L_\mathrm{OD})^{-1} \\
  & = -L_\mathrm{OD}(1-L_\mathrm{OD})^{-2}\leq 1.   
\end{aligned}
\]
On the other hand, integrating by parts in the~$p$ variable produces
\[
\begin{aligned}
\Pi\left[\partial_{p_i}\left(\nabla_p f \cdot \nabla_p g\right)\right] &
= \Pi \left[ f \nabla_p^* \cdot \left(p_i \nabla_p g \right)\right] = \Pi\left[ f (p_i p\cdot \nabla_p g - p_i \Delta_p g - \partial_{p_i}g)\right].  
\end{aligned}
\]
Putting these estimates together with Lemma~\ref{lem:standard} and $\|P_t\phi\|_\infty \leq \|\phi\|_\infty$, we can estimate $R_1(n)$ and $R_2(n)$ as
  \[
  |R_1(n)|+ |R_2(n)| \leq c_n^1 \| \phi \|_\infty \left\| \sPp P_t^n \phi \right\| \leq c_n^1 \| \phi\|_\infty^2 
  \]
  where we used the fact that $\chi_n\in [0,1]$ and introduced
  \[
  c_n^1 = \frac{1}{2} \| L \chi_n(H) \| + 2\gamma \sum_{i=1}^d \left\| p_i \left(p^2-d-1\right) \chi_n'(H) + p_i p^2 \chi_n''(H)  \right\|.
  \]
 By construction of $\chi_n(H)$ and the moment estimates of Lemma~\ref{lem:diff}, a simple argument based on dominated convergence shows that $c_n^1\rightarrow 0$ as $n\to \infty$.      
  
  Turning to the term $T_2$, a similar argument yields
  \begin{align*}
    T_2 &= \left\langle L^* A\sPp P_t^n \phi, \sPp P_t^n \phi \right\rangle - \left\langle A\sPp P_t^n \phi, P_t \phi L \chi_n(H) \right\rangle \\
    & \qquad - 2\gamma \left\langle A\sPp P_t^n \phi, \nabla_p(P_t \phi) \cdot \nabla_p  \chi_n(H) \right\rangle\\
  &= \left\langle L^* A\sPp P_t^n \phi, \sPp P_t^n \phi \right\rangle - \left\langle A\sPp P_t^n \phi, P_t \phi L \chi_n(H) \right\rangle \\
  & \qquad- 2\gamma \left\langle A\sPp P_t^n \phi, (P_t \phi) (p \cdot \nabla_p-\Delta_p)\chi_n(H) \right\rangle \\
  &= :T_2' +R_3(n) + R_4(n),  
  \end{align*}
  with $|R_3(n)|+ |R_4(n)| \leq c_n^2 \| \phi\|_\infty^2$ where $c_n^2\rightarrow 0$ as $n\rightarrow \infty$.  Finally, since $\sP \left[(LP_t \phi) \chi_n(H)\right] = \sP \left[P_t \phi L^* \chi_n(H)\right]$, the term $T_3$ can be estimated as
  \begin{align*}
  |T_3|\leq \frac{1}{2}\|P_t^n\phi \|\| \phi\|_\infty \|L^*\chi_n(H)\|\leq c_3^n \|\phi\|_\infty^2, 
  \end{align*}  
  with $c_3^n \to 0$ as $n \to +\infty$ in view of Lemmas~\ref{lem:diff} and~\ref{lem:standard}.
  
  It remains at this stage to treat the terms~$T_1'$ and~$T_2'$. Using Lemmas~\ref{lem:standard} and~\ref{lem:weight} and the coercivity of $AL_\mathrm{H} \Pi$ (granted by~\eqref{eq:coercivity_q} and extended to functions in~$L^2(d\mu)$ by a limiting argument),
\begin{align*}
T_1'+ T_2' &= \left\langle A L_\mathrm{H} \Pi \sPp P_t^n \phi, \sPp P_t^n \phi \right\rangle + \left\langle A L_\mathrm{H}(1- \Pi) \sPp P_t^n \phi, \sPp P_t^n \phi \right\rangle\\
&\qquad + \gamma \left\langle A L_\mathrm{OU}\sPp P_t^n \phi, \sPp P_t^n \phi \right\rangle + \left\langle L^* A  \sPp P_t^n \phi, \sPp P_t^n \phi \right\rangle  \\
&\leq -\frac{\rho}{\rho+1} \left\| \Pi \sPp P_t^n \phi \right\|^2 + \left(\eta_\varepsilon+\frac\gamma2\right) \left\| (1-\Pi) \sPp P_t^n \phi\right\| \left\| \Pi \sPp P_t^n \phi \right\| \\
&\qquad+ \left\| (1-\Pi) \sPp P_t^n \phi \right\|^2. 
\end{align*}
Combining these estimates, integrating from $s$ to $t$, and then taking $n\rightarrow \infty$ produces the following estimate
\begin{align*}
\left\langle A P_t \phi, P_t \phi \right\rangle \leq \left\langle A P_s \phi, P_s \phi \right\rangle & + \int_s^t \left[- \frac{\rho}{\rho+1}\| \Pi P_u \phi\|^2 + \left(\eta_\varepsilon+\frac\gamma2\right) \| (1-\Pi) P_u \phi\| \| \Pi P_u \phi\| \right] du\\
& + \int_s^t  \| (1-\Pi) P_u \phi\|^2 \, du,   
\end{align*}
where we used that~$P_t^n\phi$ converges to~$P_t \phi$ in~$L^2(d\mu)$ by the dominated convergence theorem since $P_t^n\phi$ converges to $P_t\phi$ pointwise and both quantities are bounded by $\|\phi\|_\infty$.  This concludes the proof of part~(ii).  

For part~(iii), we introduce $\psi= P_t \phi \in C_{0,\mathrm{b}}^\infty(\X)$.  For any $n\in \N$,
\begin{align}
  \frac{d}{dt} \| P_t\phi \|^2_{L^2(W^*\chi_n(H) \,d\mu)} &= 2 \int_\X \psi (L\psi) W^* \chi_n(H)\, d\mu \nonumber \\
  & = -2 \gamma \int_\X |\nabla_p \psi|^2 W^* \chi_n(H)\, d\mu + \int_\X  L\left(\psi^2\right) W^* \chi_n(H)\, d\mu. \label{eq:preparation_Gronwall_weighted} 
\end{align}
The second integral on the right hand side of the above equality can be estimated as follows since $W^*\in \mathscr{W}_{\alpha,\beta}(L^*)$:
\begin{align*}
  & \int_\X  L\left(\psi^2\right) W^* \chi_n(H)\, d\mu = \int_\X \psi^2 L^*\left[W^* \chi_n(H)\right] d\mu \\
  & \qquad = \int_\X \psi^2 (L^*W^*) \chi_n(H)\, d\mu + \int_\X \psi^2 W^* \left[L^*\chi_n(H)\right] d\mu + 2\gamma \int_\X \psi^2 \nabla_pW^* \cdot \nabla_p\left[\chi_n(H)\right] d\mu \\
  & \qquad \leq \int_\X \psi^2 (-\alpha W^* + \beta)  \chi_n(H)\, d\mu + \int_\X \psi^2 W^* \left[L^*\chi_n(H)\right] d\mu + 2\gamma \int_\X \psi^2 \nabla_pW^* \cdot \nabla_p\left[\chi_n(H)\right] d\mu.
\end{align*}
Since $W^*$ is strongly integrable, Lemma~\ref{lem:diff} implies that the last two integrals on the last line above converge to $0$ as $n\rightarrow \infty$. To see why, let us for instance show how to establish this fact for the second integral on the right hand side of the last inequality. Using $L_\mathrm{H}^* \chi_n(H) = 0$,
\begin{align*}
L^* \chi_n(H) = \gamma \left[ (-|p|^2+d) \chi_n'(H) + |p|^2 \chi_n''(H) \right].    
\end{align*}
Since $W^*$ is strongly integrable, there exist $C \in \mathbb{R}_+$ and $\delta \in (0,1)$ such that  
\[
\left|\int_\X \psi^2 W^* \left[L^*\chi_n(H)\right] d\mu\right| \leq C\|\psi\|_\infty^2 \int_\X \left(|\chi_n'(H)|+ |\chi_n''(H)|\right)(1+|p|^2) \rme^{-\delta H} \, dq\, dp. 
\]
Given the specific form of the Hamiltonian~$H$, there exists $\widetilde{C} \in \mathbb{R}_+$ such that 
\[
\left|\int_\X \psi^2 W^* \left[L^*\chi_n(H)\right] d\mu\right| \leq \widetilde{C}\|\psi\|_\infty^2 \int_\X \left(|\chi_n'(H)| + |\chi_n''(H)|\right) \rme^{-\delta H/2} \, dq\, dp. 
\]
Now, Lemma~\ref{lem:diff}(iii) ensures that $\rme^{-\delta H} \in L^1(dq\,dp)$, so that the right hand side of the above inequality converges to~0 as $n\rightarrow \infty$ by the dominated convergence theorem.  The other term is dealt with in a similar fashion.

The claimed result in~(iii) then follows by integrating~\eqref{eq:preparation_Gronwall_weighted} in time, and passing to the limit $n \to +\infty$ (using the monotone and dominated convergence theorems).
\end{proof}

\subsection*{Acknowledgements.}
D.P.H and E.C. graciously acknowledge support from grants DMS-1612898 (D.P.H.) and DMS-1855504 (D.P.H.) from the National Science Foundation. The work of G.S. in funded in part by the European Research Council (ERC) under the European Union's Horizon 2020 research and innovation programme (grant agreement No 810367), and by the Agence Nationale de la Recherche, under grant ANR-19-CE40-0010-01 (QuAMProcs). M. Gordina was funded from grant DMS-1954264 from the National Science Foundation.

\bibliographystyle{amsplain}
\bibliography{L2_contractivity_v2}

\end{document}